\documentclass[12pt,a4paper,reqno]{amsart}
\usepackage{preemble, relsize, enumerate}
\usepackage{hyperref}\hypersetup{linktocpage}

\newtheorem*{theorem*}{\textnormal{\textbf{Theorem}}}
\newtheorem*{assumption*}{\textnormal{\textbf{Assumption}}}

\date{\today}

\author{Val\'erie Berth\'e}
\address{Université de Paris, CNRS, IRIF,  Paris, France}
\email{berthe@irif.fr}

\author{Stephen Cantrell}
\address{Department of Mathematics, 
University of Warwick,
Coventry, CV4 7AL, UK}
\email{stephen.cantrell@warwick.ac.uk}

\author{Jungwon Lee}
\address{Max Planck Institute for Mathematics, Vivatsgasse 7, 53111 Bonn, Germany}
\email{jungwon@mpim-bonn.mpg.de}

\author{Mark Pollicott}
\address{Department of Mathematics, 
University of Warwick,
Coventry, CV4 7AL, UK}
\email{masdbl@warwick.ac.uk}

\title[Multidimensional statistics for generalised continued fractions]{Multidimensional statistics for finite orbits of generalised continued fractions}

\begin{document}
\maketitle

\begin{abstract}
We statistically compare the relationships between frequencies of digits in continued fraction expansions of typical rational points in the unit interval and higher dimensional generalisations. This takes the form of a Large Deviation and Central Limit Theorem, including multidimensional  results for random vectors. These results apply to classical  multidimensional  continued fraction transformations including  Brun's algorithm  and the Jacobi--Perron algorithm,  and more generally for maps    satisfying  mild contraction hypothesis on the inverse branches. We prove in  particular
that  the finite trajectories capture the generic ergodic  behaviour of  infinite trajectories. 
\end{abstract}

%\setcounter{tocdepth}{1}
%\tableofcontents

\section{Introduction}

The study of continued fraction expansions of  real numbers has a long and rich history that famously includes seminal contributions by Gauss.  An important modern aspect of this theory relates to the frequencies of the digits in the expansion, with basic estimates for typical points coming from  ergodic theorems  (See  e.g., \cite{HensleyBook}).  In this article, we describe more subtle statistical results such as Central Limit Theorems (CLT) for finite  trajectories. This allows us to show that the  expansions of  rational  numbers  capture the generic ergodic behaviour of the expansions of real numbers, as  first established in Baladi--Vallée \cite{baladi-vallee}.

More generally, we formulate natural and verifiable conditions  for a map defined on a compact subset of ${\mathbb R}^m$  to exhibit nice statistical behaviour of the random vectors associated to symbolic expansions induced by the map. Before we present the main results in Section \ref{sec.main}, we motivate them by highlighting new consequences with a concrete example. 

For $(q,r,s) \in \mathbb{N}^3$ with $q>r>s$, the Brun map in dimension $m=2$  (See \cite{Brun19,Brun20,BRUN}) yields an algorithm for finding a Greatest Common Divisor (GCD) by dividing the largest entry by the second largest entry:
\[ (q,r,s) \longmapsto (q_1,r_1,s_1) \]
with $q_1,r_1,s_1 \in \{q-\lfloor q/r \rfloor r, r, s\}$ reordered in a descending order.  This map is the integer version of  Brun's multidimensional continued fraction algorithm, 
playing the same role  as  Euclid algorithm for continued fractions, and has been  studied  for   several applications (See  \cite{Ro,Wen2018}). This  algorithm terminates uniquely in finite time when it reaches $(q_n, 0, 0)$ (where $q_n$ is the GCD of  the initial  triple), and gives a sequence of partial quotients $a_i=\lfloor q_{i-1}/r_{i-1} \rfloor \geq 1$ for $1\leq i \leq n$ where $(q_0,r_0,s_0):=(q,r,s)$. For $j \in \mathbb{N}$ let
\[ N_j(q,r,s):= \#\{1 \leq i \leq n: a_i=j \}  \] 
be the number of occurrences of the prescribed partial quotients in the execution. This can be extended to $m>2$ in a canonical way. For Brun's algorithm, our result can be informally stated as follows. 
\begin{theorem} \label{motif:brun}
For $(j_1, \cdots, j_d) \in \N^d$, consider a random vector $\overline{N}=(N_{j_1}, \cdots, N_{j_d}) \in \N^d$. Then, up to suitable normalisations, the values
\[ \{ \overline{N}(t_1,\cdots, t_{m+1}): 1\leq t_{m+1}< \cdots <t_1 \leq Q , (t_1, \cdots ,t_{m+1})=1\} \]
become equidistributed according to a large deviation theorem and a Gaussian limit law as $Q \rightarrow \infty$.
\end{theorem}

The motivation for our general results is to provide an easily applicable theory that can be used to study the multidimensional statistics of finite orbits of generalised continued fraction algorithms in higher dimensions. Under mild hypothesis, we significantly simplify and generalise the transfer operator methods initially formulated in \cite{baladi-vallee} for the Gauss map, where they required a delicate analysis on uniform polynomial decay of operator norms. Verifying that these estimates hold is challenging as it needs precise knowledge of the underlying dynamics, which need to be checked case by case for different maps.  We discuss this point further in Section \ref{sec.overview}.

There are many algorithms, especially in higher dimensions for which a priori assumptions and estimates are difficult to show and have not been obtained. These include the very well-known Brun and Jacobi--Perron algorithms. In fact, there were no previous results known on the distribution  laws for their finite orbits. This has lead us to formulated a flexible framework with great generality that can be applied to a variety of continued fraction maps, including  the Brun and Jacobi--Perron algorithms. Since our hypotheses are easily  verified we expect further applications to follow in the future (See Section \ref{subsec:rem}).

\medskip

We now move on to discuss our general results and approach. 

\subsection{The setting}
Let us introduce the necessary notation to present the dynamical formulation of our main results. 

Let $I \subset \mathbb{R}^m$ be a compact connected  subset and $T$ be a self-map on $I$.
 Let $\{I_j\}_{j \in J}$ be a  countable partition (modulo a null set) of pairwise disjoint open subsets such that 
\[ I=\bigcup_{j \in J} \overline{I_j} \quad  \mbox{and }  \quad  T|_{I_j}: I_j \rightarrow T(I_j) \ \mbox{is bijective} . \] 
We further assume that  the  restriction  $T|_{I_j}$ can be extended to a $C^{1+\text{Lip}}$ map on  its  closure $\overline{I_j}$. However, we do not necessarily assume that $T$ is a full branch map, i.e., we do not require  that $T(\overline{I_j}) = I$ for all $j \in J$. 

Let $\mathcal{H}=\{ h_j=T|_{I_j}^{-1}: j \in J \}$ denote the set of inverse branches  of the map $T$. For $n \geq 1$ and $(j_1, \cdots, j_n) \in J^n$, we inductively define
\[ I_{j_1, \cdots, j_n}=\{ x \in I_{j_1}: T(x) \in I_{j_2, \cdots, j_n} \} . \]
It follows that $h=h_{j_1} \circ \cdots \circ h_{j_n}$ induces a bijection  from $T^n (I_{j_1, \cdots, j_n}) $ onto $ I_{j_1, \cdots, j_n}$. We call $h$ an \emph{admissible inverse branch of depth $n$} and denote by $\mathcal{H}^n$ the set of all such inverse branches. We also set $\mathcal{H}^\ast = \cup_{n\ge 1} \mathcal{H}^n$.

Throughout, we make the following uniform contraction, distortion and Markovian assumptions on inverse branches. We emphasise that these correspond to natural, mild conditions that are sufficient to deduce the existence of an absolutely continuous invariant measure (See e.g., \cite{mayer}).

\begin{assumption}[\textbf{Markov}]  \label{assmp:exp}
Let $(I,T)$ be as above. We assume that the map $T$ is \emph{topologically transitive} (i.e., there is a point with a dense orbit)
and that it satisfies the following.
 \begin{enumerate}
\item The inverse branches  are \emph{uniformly contracting}, i.e., there is a uniform constant $0<\rho<1$ such that   for all  $ h \in  \mathcal{H}^n $ and for $n \geq 1$
\[  |h(x)-h(y)| \leq \rho^n |x-y|,  \ x,y \in \mathrm{Dom}(h) \]
where $ \mathrm{Dom}(h)$ denotes the domain of $h$.

\item The  inverse branches have \emph{bounded distortion} of their Jacobians, i.e., there is a uniform constant $M >0$ such that    for all  $ h \in  \mathcal{H}^n $ and for $n \geq 1$

\[ \left| \frac{J_h(x)}{J_h(y)}-1\right| \leq M |x-y| , \ \ x,y \in \mathrm{Dom}(h) . \]
Here $|J_h(x)|=|\det Dh(x) |$ denotes the Jacobian determinant of $h$. %where $DT$ denotes the derivative of $T$.

\item There is a finite partition (modulo a null set) $\mathcal{P}=\{P_a\}_{a \in A}$ of pairwise disjoint open subsets for $I$ such that $I=\bigcup_{a \in A} \overline{P_a}$ and $\{P_a\}_{a \in A}$ is compatible with $\{I_j\}_{j \in J}$ as follows. Each $T(I_j)$ is a  union of the closure of partition elements. Further, if $I_j \cap P_a \neq \emptyset$, then $I_j \subset P_a$. %\footnote{\blue This is because open set assumption may be a problem for $\{P_a\}$ in some examples; When single $I_j$ can intersect several $P_a$'s simultaneously, or not?}  
Lastly, for $j \in J$ and $a \in A$, there either exists a unique $b \in A$ such that 
\[ h_j(P_a) \subset P_b \cap I_j, \ \mbox{or } h_j(P_a) \cap I = \emptyset . \]
\end{enumerate}
\end{assumption}

We now describe the subspace of $I$  on which our estimates apply  in terms of  finite backward orbits under the map $T$ of a given point.  Fix a point $x_0 \in I$ and write
\begin{equation}\label{eq:defX}
X = \bigcup_{n \ge 1} X_n \ \text{ where } \ X_n = \{ x \in I : x =  h(x_0) \text{ for some } h \in \mathcal{H}^n\}.
\end{equation}
 Given an element $x=h(x_0) \in X$ with $h \in \mathcal{H}^\ast$, we assign to it the \emph{weight}
\begin{equation}\label{eq:weight}
w(x) = - \log|J_h(x_0)|.
\end{equation}
We impose a further assumption, called  \emph{non-arithmeticity}.
\begin{assumption} [\textbf{Non-arithmeticity}]\label{assump:nonarith}
We assume that the function $w$ does not take values lying in a lattice. That is, there does not exist $c  \in \R$ such that
\[
\{ w(x) = -\log |J_h(x)| : x \in I, h \in \mathcal{H}^\ast \text{ and $h(x) = x$}\} \subset c \Z.
\]
\end{assumption}
A familiar motivating example is provided by the Gauss map. 

\begin{example} \label{ex:intro:gauss}
Let  $I=[0,1]$ and  let $T:[0,1] \to [0,1]$ denote the Gauss map 
defined by 
$$
T(x) =
\begin{cases}
\frac{1}{x} - \lfloor \frac{1}{x}\rfloor = \{\frac{1}{x}\}  & \hbox{ if } 0 < x \leq 1\\
0 &  \hbox{ if } x=0 .
\end{cases}
$$

For $j \in \mathbb N_{\geq1}$, set $I_j:=(1/(j+1), 1/j)$.  With the notation from Assumption \ref{assmp:exp}, the set $J$  equals $\mathbb N$, and  in this case $\mathcal P$ consists of a single element $(0,1)$. We have $I=\bigcup_{j \in J} \overline{I_j}=\bigcup_{P \in \Pc} \overline{P}$ with $\mathrm{Leb}(\partial I_j)=\mathrm{Leb}(\partial P)=0$.  The inverse branch of $T$ associated with $j \in J$  is $h_j: (0,1) \to I_j$ given by $x \mapsto \frac{1}{j+x}$.

 In the particular case that $x_0=0$, we can identify  the set $X $ from (\ref{eq:defX}) as $X= \mathbb Q$.  Moreover, for any $ n \geq1$, the image of $x_0$ under the elements in $\mathcal{H}^n$  correspond to rationals $p/q $ in $I$  with $(p,q)=1$ and for which the Euclidean algorithm stops   exactly after  $n$ steps.
 One gets a unique expansion 
 with the convention that the  last partial quotient in the corresponding  continued fraction expansion is  larger than or equal to $2$. More precisely, 
  for $x \in (0,1) \cap \mathbb{Q}$, we have  $j_1, \cdots, j_{n-1} \geq 1$ and $j_n \geq 2$ such that for all $i \in \{1, \cdots, n-1\}$
\[ T^{i-1}(x) \in I_{j_i}, \  T^{n-1}(x) \in \overline{I_{j_n}}, \  \mbox{ and } T^n(x)=0 .\]
 This yields
\[ x:=p/q=[0; j_1, \ldots, j_n] = h_{j_1} \circ \cdots \circ h_{j_n} (0).
\]

Note that we can identify $h \in \mathcal{H}^n$ with the linear fractional transformation 
\[
x \mapsto \frac{rx+p}{sx+q}  \, \text{ and the matrix } \, \begin{bmatrix} r & p \\ s & q \end{bmatrix} =\begin{bmatrix} p_{n-1} & {p_{n}}\\ q_{n-1} & {q_n} \end{bmatrix}=\begin{bmatrix} 0 & 1 \\ 1 & j_1 \end{bmatrix} \cdots \begin{bmatrix} 0 & 1 \\ 1 & j_n \end{bmatrix} \in \GL_2(\Z).
\]
 Thus we have $q=|(h_{j_1} \circ \cdots \circ h_{j_n})'(0)|^{\frac{1}{2}}$ and the associated weight becomes $w(p/q) = 2 \log q$.
We will check in Section \ref{subsec:cfGauss} that Assumptions \ref{assmp:exp} and \ref{assump:nonarith} are satisfied.
\end{example}

We aim to study distributional results for $h(x_0)$, for $h \in \mathcal{H}^n$. As far as we know, there are only a few examples where related results have been established, which include the work \cite{baladi-vallee} where they successfully adopted transfer operator analysis for the limit laws of the costs associated to continued fractions on the set of rationals with bounded denominator. 
Their  method may generalise to other settings, and does have error estimates, but  unfortunately, their proof is quite technical and requires strong hypothesis. In addition to Assumptions \ref{assmp:exp} and \ref{assump:nonarith}, their arguments need  topological mixing and a Uniform Non-Integrability (UNI) property for Dolgopyat estimate that can be very difficult to verify in other cases.

\subsection{The results}\label{sec.main}
The purpose of this article is to present a general ``Dolgopyat estimate free" proof that the even more general vector-valued costs associated to  dynamical system $(I,T)$ satisfying Assumption \ref{assmp:exp} and \ref{assump:nonarith} admit a large deviation principle and  
a Gaussian limit distribution as random variables on a discrete subspace $X$ of $I$. 

This not only allows us to have the basic CLTs for finite backward  orbits simply by checking a priori mild assumptions, but also to observe interesting multidimensional statistics.  In particular, we compare frequencies of digits in the symbolic expansions in terms of the associated covariance matrix.

We assign to  each $x \in I$ a cost  via the locally constant functions $c_j: I \to \mathbb{R}_{\geq 0}$  defined for $j \in J$ by 
$$
c_j (x)= 
\begin{cases}
1 &\hbox{ if } x \in I_j \cr
0 & \hbox{ otherwise.} 
\end{cases}
$$
In this definition, the weight $1$ can be replaced by any $c>0$, however we keep $c=1$ to explicitly remark later on the genericity for the behaviour of finite orbits.

We define  a \emph{counting function} $N_j$ on $X$ in a natural way: If $x = h(x_0)$ with $h \in \mathcal H^n$, then we write
\[
N_j(x) = \#\{ 0\le i \le n-1 : T^{i}(x) \in I_j \}= \sum_{0\le i \le n-1} c_j( T^{i}(x)).\]
  %Note that if $x=h(y)$, with $ h \in   \mathcal H$, then  $ N_j(x)=c_j(x)=1$.

From now on, we assume  $(I,T)$ and $w(\cdot)$ satisfy Assumptions \ref{assmp:exp} and \ref{assump:nonarith}. Our first result describes a frequency result for  elements of $X$ with respect to the non-arithmetic weight $w$.
The following law of large numbers shows in particular that the set $X$ is not  too thin.

\begin{theorem}\label{thm.firstmoment}
There  exists $C>0$ such that $$\#\{x \in X : w(x) < Q\} \sim C e^{Q}$$ as $Q \to \infty$.
Moreover,
for each $j \in J$, there exists $\Lambda_j > 0$ such that
\[
\lim_{Q\to\infty} \frac{1}{\#\{x \in   X : w(x) < Q\}}\sum_{w(x) < Q}  \frac{N_j(x)}{Q}  = \Lambda_j. 
\]
In particular, we have \[
\sum_{w(x) < Q}  N_j(x)  \sim C \, \Lambda_j  \, Q\, e^{Q}
\mbox{ as }Q \to \infty.\]
\end{theorem}
Here and throughout the rest of the paper, the notation `$\sim$' represents the asymptotic limit: that is if $f,g : \R \to \R$ are real valued functions then $f(t) \sim g(t)$ as $t \to\infty$ means that $f(t)/g(t) \to 1$ as $t\to\infty$.

\begin{example} \label{expl:gauss}
 We revisit Example \ref{ex:intro:gauss}  with 
 the Gauss map. %Let us  take $c=1$  in (\ref{eq:cost}). 
 Writing 
 $$x = [0; j_1, \ldots, j_{n}] = \frac{1}{j_1 + \frac{1}{j_2 + \cdots + \frac{1}{j_n}}}$$
 as a finite continued fraction expansion, this  yields  $ N_j(x) = \#\{1 \leq i \leq n \hbox{ : }  j_i = j \}$, i.e., $  N_j(x)$ counts the number of
  occurrences of the prescribed partial quotient  $j$  in the expansion of $x$.  We will see that $\Lambda_j$ is given by $ \int_{I_j} g(x)dx$ (up to an absolute positive constant $2/h(T)$,
  where $h(T)$ is the entropy of the Gauss map),   and $g(x)=\frac{1}{\log 2 (1+x)}$ is the invariant density.  
  
%We can compare this with the classical result of L\'evy from 1937 who showed that for Lebesgue typical points $x$  the frequency of the digit $q$ in the continued fraction expansion is $\Lambda_j:= \frac{1}{\log 2} \log \frac{(q+1)^2}{q (q+2)}$. \footnote{\color{blue} For typical points, this is given by the formula $\sum_{p \in J} N_j(h_p(x)) \int_{I_p} g(x)dx$, where $g$ is the invariant density. Thus (guess) we will have the explicit value of $\Lambda_j$, where we have the explicit form of the density function.} %\footnote{Is there a reference for the usual CLT for the Central Limit Theorem?} 
\end{example}

In fact, we can  also deduce  Theorem \ref{thm.firstmoment} from the following large deviation principle and central limit theorem.
\begin{theorem} \label{thm.largedev}
For each $j \ge 1$ and  for any $\epsilon >0$
\[
\limsup_{Q\to\infty} \frac{1}{Q} \log \left( \frac{1}{\#\{x \in X : w(x) < Q\}} \#\left\{ x\in X: w(x) < Q, \, \left| \frac{N_j(x)}{Q} - \Lambda_j \right| > \epsilon \right\}\right) < 0.
\]
Furthermore, for each $j \ge 1$ there exists $\sigma^2_j > 0$ such that for any $a \in \R$
\[
\frac{1}{\#\{x \in X: w(x) <Q\}} \#\left\{ x \in X: w(x)< Q \ \text{ and } \frac{N_j(x)}{\sqrt{Q}} \le a \right\} \to \frac{1}{\sqrt{2 \pi} \sigma} \int_{-\infty}^a e^{-u^2/2\sigma_j^2}\ du.
\]
\end{theorem}

\medskip

Let us now turn to the multidimensional version. Let $d\geq 1$ and fix  distinct numbers $j_1, \ldots, j_d $ in $J$ that are
admissible in the sense of Assumption \ref{assmp:exp}.(3). We now want to compare the counting  functions $N_{j_1}, \ldots, N_{j_d}$. To simplify the notation, we write
\begin{equation}\label{eq:nbar}
\overline{N}(x) = (N_{j_1}(x), \ldots, N_{j_d}(x)) \in \N^d \ \text{ for $x \in X$.}
\end{equation}
We  also consider the \emph{centred  counting functions}
\[
\varphi_i(x) = N_{j_i}(x) - w(x)  \Lambda_{j_i} 
\]
and write
\begin{equation}\label{eq:phibar}
\overline{\varphi}(x) = (\varphi_1(x), \ldots, \varphi_d(x)) \in \mathbb{R}^d.
\end{equation}

We then have the following Central Limit Theorem, which formalises Theorem \ref{motif:brun} for general maps satisfying Assumption \ref{assmp:exp}  and \ref{assump:nonarith}.

\begin{theorem}\label{thm.mdclt:intro}
There exists a positive definite, symmetric matrix $\Sigma \in \GL_d(\mathbb R)$ such that, for any non-empty open $U \subset \mathbb R^d$, we have
\[
\frac{1}{\#\{x \in X: w(x) <Q\}} \#\left\{ x \in X: w(x)< Q \ \text{ and } \frac{\overline{\varphi}(x)}{\sqrt{Q}} \in U \right\} \to \frac{1}{(2 \pi \det(\Sigma))^{d/2}} \int_{U} e^{-\frac{1}{2} \langle u, \Sigma u \rangle} \ du
\]
as $Q\to\infty$.
Moreover, the covariance matrix $\Sigma$ is described by a Hessian of $\overline{\varphi}$.
\end{theorem}

\begin{example}
For the Gauss map, we have $T(I_j)=(0,1)$ for all $j \in J$. This means that there is no correlation among the  admissible digits, hence the covariance matrix $\Sigma$ is given by a diagonal matrix. In fact, the covariance matrix encodes explicitly the relations provided by the Markovian condition from Assumption \ref{assmp:exp}.(3).
\end{example}

As an immediate consequence of Theorem \ref{thm.largedev}, we also deduce the following multidimensional large deviation principle.
\begin{theorem} \label{thm.multilargedev}
For each $\epsilon >0$
\[
\limsup_{Q\to\infty} \frac{1}{Q} \log \left( \frac{1}{\#\{x \in X : w(x) < Q\}} \#\left\{ x\in X: w(x) < Q, \,  \left\| \frac{\overline{\varphi}(x)}{Q} \right\| > \epsilon \right\}\right) < 0
\]
where $\|\cdot\|$ is any fixed norm on $\R^d$.
\end{theorem}

We remark that  we are counting the points in $X$ up to finite multiplicity in Theorem \ref{thm.firstmoment}-\ref{thm.multilargedev} due to boundary issues for the function space we take in Definition \ref{def:lip} if $|\mathcal{P}| >1$. To remove this restriction, we may fix a point $x_0 \in I \backslash \{ T^n(x): x \in \partial P_a, P_a \in \mathcal{P}, n \geq 1 \}$. 

For an arbitrary choice of the initial point $x_0$, we believe that this issue can be resolved by adapting \cite{kim2025}. They suggested a way of introducing a finite CW-cell structure on the space of piecewise $C^1$-functions on $I$ with respect to $\mathcal{P}$ to deduce the exact counting result for the cost functions associated to complex continued fractions (See Remark \ref{rmk:cell}). 

\subsection{An overview of the proofs}\label{sec.overview}

We now briefly describe our approach and compare it with earlier approaches. The original argument employed
 by Baladi--Vall\'ee holds for the Gauss map when $m=d=1$.
They applied a Tauberian theorem (Perron's formula of order 2) to complex Dirichlet series identified using the resolvent of transfer operators, and accordingly obtained a ``Quasi-power behaviour" of the moment generating function of centred counting functions $\overline{\varphi}$  from (\ref{eq:phibar}) associated to classical continued fractions.  

The uniform quasi-power expression is useful because of the Hwang's Quasi-power Theorem \cite{Hwang}.
\begin{theorem}[Hwang's Quasi-power Theorem]
Suppose that a sequence of random variables $Y_n$ has the property that the moment generating functions $\lambda_n(s) = \E(e^{sY_n})$ are analytic in a disc $|s| < \rho$ for some $\rho >0$. Further assume that there is a uniform expression
\[
\lambda_n(s) = e^{\beta_n  U(s) + V(s)}\left( 1 + O\left(\frac{1}{k_n} \right)\right)
\]
where $\beta_n, k_n$ are sequences tending to $\infty$, $U,V$ are analytic in $|s| \le \rho$ and $U''(0) \neq 0$.

Then, the distribution of $Z_n :=  (Y_n - \beta_n U'(0))/ \beta_n U''(0)$ is asymptotically Gaussian, the speed of convergence to the Gaussian limit being $O(k_n^{-1} + \beta_n^{-1/2})$.
\end{theorem}

In order to deduce the quasi-power estimate, one has to apply sophisticated Tauberian theorems that provide not just an upper bound, but also an asymptotic of precisely the right order. The use of Perron's formula  in  \cite{baladi-vallee} requires finer analysis on a Dirichlet series, in particular, uniformity and analytic continuation to $\Re(s)\geq 1-\epsilon$ with no other pole on the vertical line $\Re(s)=1$, except a unique pole at $s=1$ of order 1. To have so,  strong spectral properties are needed, namely  that the transfer operator satisfies the Dolgopyat estimate and has a spectral gap with a unique dominant eigenvalue that is simple.

%\tcb{Explain the difficulty in applying.}

%We now describe the approach. The original argument made in Baladi--Vallée is to apply a Tauberian theorem (Perron's formula of order 2) to a Dirichlet series given by the resolvent of transfer operator, accordingly to obtain a ``Quasi-power behavior" of the moment generating function of $\widehat{\varphi}$ associated to the continued fractions (with $d=1$), which implies the limit Gaussian distribution with optimal error terms. To do so, they needed nice spectral properties that the transfer operator has a spectral gap with the unique largest eigenvalue and satisfies the uniform Dolgopyat estimate. 

Here, for all $(I,T)$ under Assumption \ref{assmp:exp} and \ref{assump:nonarith}, it is possible to deduce a spectral gap on the piecewise Lipschitz space $\oplus_{a \in A} C^{1+\mathrm{Lip}}(P_a)$ possibly without the uniqueness of the largest eigenvalue and  without 
a Dolgopyat-type estimate. Motivated by Morris \cite{morris}, this is sufficient to have Theorem \ref{thm.mdclt:intro} by generalising the method of moments. However, we remark that our result does not give any information regarding the speed of convergence, in contrast to the Baladi--Vall\'ee theorem \cite{baladi-vallee}. 

This article is organised as follows. In Section \ref{sec:examples}, we first present applications of the main theorems to some examples of multidimensional continued fraction algorithm by checking the a priori hypotheses. In Section \ref{sec:dynamics}, we introduce complex functions and transfer operators, and study their spectral properties. 
Then the law of large  numbers is proved in Section~\ref{subsec:LL} with the proof of Theorem \ref{thm.firstmoment}. 
Theorem \ref{thm.largedev}  on large deviations is proved in Section \ref{subsec:LD}. This leads to CLTs based on the moments estimates in Sections \ref{subsec:CLT}
and \ref{subsec:CLT2}.

\subsubsection*{Acknowledgements}
We  warmly thank Florian Luca for fruitful discussions  on establishing the  non-arithmeticity condition. The first author is supported by Agence Nationale de la Recherche through the project SymDynAr (ANR-23-CE40-0024) and 2024 ERC Synergy Project DynAMiCs (101167561). The last named author is supported by ERC-Advanced Grant 833802-Resonances, EPSRC Grants: APP29916 and EP/W033917/1. 

%%%%%%%%%%%%%%%%%%%%%%%%%%%%%%%%%%%%%%%%%%%%%%%%%%%%%%%%%%%%%%%%%%

\section{Application: Examples} \label{sec:examples}

In this section, we present CLTs and multidimensional statistics for the Gauss map and generalised continued fractions simply by checking the hypotheses in Assumptions  \ref{assmp:exp} and \ref{assump:nonarith}.
Our examples include classical continued fraction (dimension $m=1$),  Brun and Jacobi-Perron  multidimensional continued fractions ($m \geq 2$). % {\color{blue} Explain below the meaning of $N_j$ and multi-dim CLT for each example.}

\subsection{Continued fractions}\label{subsec:cfGauss}

In Example \ref{ex:intro:gauss}, we have seen that each $x \in (0,1) \cap \mathbb Q$ admits a unique continued fraction expansion 
\[ x=[0; j_1, \ldots ,j_n] \]
with $j_1, \cdots, j_{n-1} \geq 1$, $j_n \geq 2$, which is identified by a depth $n$ inverse branch $h=h_{j_1} \circ \cdots \circ h_{j_n} \in \mathcal{H}^n$ of the Gauss map $T$ on $I=[0,1]$. 

We now check that $(I, T)$ satisfies all of the required hypotheses from Assumption  \ref{assmp:exp} and \ref{assump:nonarith}.

\begin{proposition}
The Gauss map satisfies the Markov  and Non-arithmeticity  assumptions.
\end{proposition}

\begin{proof}
 For all positive integer $j$ and $I_j=(\frac{1}{j+1}, \frac{1}{j})$, we have $T(\overline{I_j})=I$ which immediately shows that $T$ is topologically mixing, so transitive. 

Using the explicit form  $h_j(x)=\frac{1}{j+x}$, it is not difficult to see the following by induction. For all $n \geq 1$ and $h \in \mathcal{H}^n$, we have for all $x \in I$ 
\[ |h'(x)| \ll \left(\frac{1}{2} \right)^n \ \mbox{and } \ \frac{|h''(x)|}{|h'(x)|} \ll M \]
for some $M>0$, which imply the Markov condition from Assumption \ref{assmp:exp} by simple applications of the mean value theorem.
One can also use the explicit form  of $h \in \mathcal{H}^n$, that is, 
\[
h(x)=\frac{p_{n-1}x+p_n}{q_{n-1}x +q_n}, \, \text{ where $p_n, q_n$ satisfy } \, \begin{bmatrix} p_{n-1} & p_n \\ q_{n-1} & q_n \end{bmatrix} =\begin{bmatrix} 0 & 1 \\ 1 & j_1 \end{bmatrix}  \cdots \begin{bmatrix} 0 & 1 \\ 1 & j_n \end{bmatrix},
\] 
which gives  $h'(x)= \frac{-1}{(q_{n-1}x +q_n)^2}$, that is, the derivative of  the (unimodular) homography $h$ is  provided by minus the  inverse of the square of  its  denominator.   (This is  a crucial property providing  a simple expression  of the  derivatives  of  the inverse branches in terms of denominators  of the  corresponding homographies,    that will also hold in the higher dimensional case.)

For the non-arithmeticity condition from Assumption \ref{assump:nonarith},  it is sufficient to consider  the quadratic numbers    with  purely periodic continued  fraction expansions   $ \frac{\sqrt 5-1}{2}= [0;1,\ldots, 1 ]$
and $ \sqrt 2-1= [0;2,\ldots, 2 ]$, which are respectively fixed points of the  inverse branches   $h_1= (1+x)^{-1}$ and $h_2= (2+x)^{-1}$,
which gives 
\[
\log \left| h_1'\left( \frac{\sqrt 5-1}{2}\right)\right|= -2 \log \left(\frac{\sqrt 5+1}{2}\right) \, \text{ and } \,  | h_2'( \sqrt 2-1)|= - 2 \log (1+\sqrt 2).
\]
We can then easily see that $\log  (1+ \sqrt2)/\log ((1+\sqrt 5)/2) \not  \in {\mathbb Q}$ concluding the proof.
\end{proof}

The quantity $N_j(x)$ counts  the number of occurrences of the prescribed partial quotient $j$ in the continued fraction expansion for
 all $x \in X$. Theorem \ref{thm.firstmoment} then gives
\[
\lim_{Q\to\infty} \frac{1}{\#\{p/q : 2\log q < Q\}} \sum_{2\log q < Q} \frac{N_j(p/q)}{2\log q}  = \Lambda_j,
\]
where $$\Lambda_j=  \frac{\int_{I_j} g(x)dx}{h(T)}= \frac{2}{h(T)\log 2} \log \frac{(j+1)^2}{j (j+2)}$$ and $g(x)=(\log 2 (1+x))^{-1}$ denotes the Gauss  invariant density for $(I,T)$. %
Indeed one has $$-\widetilde{\lambda}_s(1,0)=\int_I  \log | T'(x)|   g(x) dx=h(T)= \frac{\pi^2}{ 6 \log 2}$$ by Rohlin's formula  for  the entropy $h(T)$ of the Gauss map $T$,  and $$\widetilde{\lambda}_t(1,0)=\int_{I_j} g(x) dx=\frac{1}{\log 2} \log \frac{(j+1)^2}{j (j+2)}.$$

We recover the classical result on the number of steps in Euclid algorithm.
Moreover we can compare this with the classical result of L\'evy from 1937 who showed that for Lebesgue typical points $x$  the frequency of the digit $j$ in the continued fraction expansion is $\frac{1}{\log 2} \log \frac{(j+1)^2}{j (j+2)}$. 

%\tcb{Rational numbers are counted twice.; Corrected with final operator.}

Theorem \ref{thm.mdclt:intro} gives that for the distinct $j_1, \ldots, j_d$, a random vector $\overline \varphi=(\varphi_1, \ldots, \varphi_d)$, where each $\varphi_i(x) = N_{j_i}(x) - w(x)\Lambda_{j_i} $, has limit Gaussian distribution with the covariance matrix  given by the identity matrix. This is because there is no correlation among admissible digits, i.e., we have $T(\overline{I_j})=I$ for all $j \in J$.

\subsection{Multidimensional continued fractions}  \label{subsec:mcf}
 We consider unimodular continued fraction algorithms such as  defined in \cite{Lagarias:93}; See also \cite{Schweiger:00}.  These algorithms 
associate with some  given vector   an infinite sequence of matrices with determinant $\pm$1, and one can consider the quality of convergence  of this product of matrices which correspond to the contraction property in Assumption \ref{assmp:exp}.(1).
 In this setting, the inverse branches  are  homographies  (as in the regular continued fraction case) and the Jacobian determinant has a simple formulation in terms of the denominator of the 
 homography (See e.g., \cite[Proposition~5.2]{Veech:78}), which allows one  to deduce easily   the bounded distortion in Assumption  \ref{assmp:exp}.(2). 
 % %Condition (1) is often satisfied. Bounded distorsion is a classical property for interval maps  and their  generalizations that  provide the existence of an ergodic invariant absolutely continuous probability measure.  See \cite{Renyi:57} in the setting of $f$-expansions.

\subsubsection{Brun's algorithm}\label{ex:Brun}
We consider
here  the continued fraction   version of the Brun GCD algorithm discussed  in the introduction.
This is   one of  the  most classical  multidimensional continued fraction algorithms (See \cite{Brun19,Brun20,BRUN,Schweiger:00} and  \cite{BLV:18} for  mean results  on rational trajectories). The $m$-dimensional \emph{Brun algorithm}  $T_{\rm B}:[0,1]^m \to [0,1]^m$  is
defined  for  $(x_1, \ldots, x_m) \in [0,1]^m$  
by
$$
T_{\rm B}(x_1, \ldots, x_m)
=\begin{cases}  \left(  \frac{x_{i+1}}{x_i}, \ldots, \frac{x_m}{x_i }, \ \left\{ \frac{1}{x_i}\right\}, \frac{x_{1}}{x_i}, \ldots,  \frac{x_{i - 1}}{x_i}
\right)& \hbox{ if } x_i \neq 0\\
 (0,\ldots, 0)&  \hbox{ otherwise},
 \end{cases}
$$
where $x_i = \max_{k} \{x_k\}$. Its density function is  explicitly given in \cite{Schweiger79,Arnoux-Nogueira} by 
\[ 
\sum_{\sigma \in \mathfrak{S}_{m}} \prod_{i=1}^{m}  \frac{1}{1+x_{\sigma(1)}+\ldots+x_{\sigma(i)}},
\]
where $ \mathfrak{S}_m$   is the set of permutations on  $m$ elements. This transformation  satisfies  all of the required  hypotheses. 

\begin{proposition}
Brun's algorithm satisfies the Markov and Non-arithmeticty assumptions.
\end{proposition}

\begin{proof} The  partition $(I_j)$  is  indexed  by the  set  $J=\N$ of positive integers.
Inverse  branches in ${\mathcal H}$  are homographies  of the  form 
\begin{equation}\label{eq:Brunh}
 \left( \frac {x_{i+1}} {j+ x_i},
 \ldots,   \frac {x_{m}} {j + x_i},  \frac 1 {j + x_i}, \frac {x_{1}} {j + x_i}, \ldots,   \frac {x_{i-1}} {j + x_i}\right)
 \end{equation}
 where $x_i = \max_{k} \{x_k\}$ and $j =\lfloor 1/x_i\rfloor$. This is a full branch algorithm so that $\mathcal{P}$ consists of a single element $(0,1)^m$. Transitivity follows immediately. The contraction  is  proved by  exhibiting using
  the Hilbert distance a suitable  metric that is  contracted by the homographies
(See \cite[Annexe]{BAG:01}).   Distortion properties are proved thanks to the  expression of the  Jacobian as 
$|J_h(x)|= \frac{1}{(j+x_i)^{m+1}}$  for $h$ as in  (\ref{eq:Brunh}).

%This is   the multiplicative form of the unordered version of Brun algorithm.

For the  non-arithmeticity, consider  the  unique root  $\tau_m$   with  $0< \tau_m <1$  of $x^{m+1}+x-1=0$, and  the unique root $\varrho_m$ 
with  $0< \varrho_m <1$  of $x^{m+1}+2x-1=0$,
for $ m \geq 2$. The algebraic number $(\tau_m^m ,  \tau_m^{m-1}, \ldots, \tau_m)$ has  purely periodic expansion. Indeed one has 
 $$T_B (\tau_m^m ,  \tau_m^{m-1}, \ldots, \tau_m)= (\{1/\tau_m\},\tau_m{m-1},   \ldots, \tau_m)= (\tau_m ^m ,  \tau_m^{m-1}, \ldots, \tau_m).$$
  Similarly,  one checks  that  $(\varrho_m^m ,  \varrho_m^{m-1}, \ldots, \varrho_m)$
  has   purely periodic expansion. One then  observes  that 
  $\log(  1 + \tau_m)$ and   $\log(  2 + \varrho_m)$ are rationally independent. We conclude by  recalling that the  Jacobians  are  given  by  powers of the denominators of the homographies in $\mathcal{H}$. 
\end{proof}

\subsubsection{Jacobi--Perron algorithm}  \label{subsec:JP}
The  Jacobi--Perron algorithm   is  also one of the most famous multidimensional continued fraction algorithms;  See \cite{Bernstein:71,Heine1868,Perron:07,Schweiger:73} or ~\cite[Chapter~4]{Schweiger:00}.  %We extend here significantly the results of \cite{Waterman1975} which consider certain probabilistic behaviour of rational expansions. 
There is no known     expression of    the  absolutely continuous invariant measure  of the  Jacobi--Perron algorithm, however its (piecewise) analyticity has been established in \cite{broise}. However note that with our  approach  we do not need to have  an explicit  form  of  its absolutely continuous invariant measure.

Let $ m \geq 2$.  The $m$-dimensional \emph{Jacobi-Perron algorithm}  $T_{\rm JP}:[0,1]^m \to [0,1]^m$  is
defined as

$$
T_{\rm JP}(x_1, \ldots, x_m)
=\begin{cases}  \left( \left\{ \frac{x_2}{x_1} \right\},
 \left\{ \frac{x_3}{x_1}\right\}, \ldots , 
 \left\{ \frac{1}{x_1}\right\}\right)  & \hbox{ if } x_1 \neq 0\\
 (0,\ldots, 0)&  \hbox{ otherwise}.
 \end{cases}
$$
The inverse branches are of the form 

$$
 \left( \frac {1} {a_m+ x_m}, \frac {x_1+a_{1}} {a_m+ x_m},
 \ldots,   \frac {x_{m-1}+a_{m-1}} {a_m + x_m}\right)$$
 where $a_i =  \lfloor x_{i+1}/x_1 \rfloor $ for $1 \leq i \leq m-1$,  and $a_m=  \lfloor  \frac{1}{x_1} \rfloor $.

We check that this transformation  satisfies  all of the required hypotheses.

\begin{proposition}
The Jacobi-Perron algorithm satisfies the Markov and Non-arithmeticty assumptions.
\end{proposition}
\begin{proof}
This algorithm is not a full branch map but is topologically mixing since we have $T_{\rm JP}^m(\overline{I_j})=I$, where the partition is indexed by the proper subset $J$ of $\Z_{\geq 0}^m$ that is described below. Hence the transitivity follows immediately.

The contracting property from Assumption \ref{assmp:exp}.(1) is given in  \cite{broise}; See also  \cite{BAG:01}.  
To check (2), one  again uses the fact that the Jacobian determinant has  a convenient  expression since  we have homographies.
Regarding (3), the map $T_{\rm JP}$ satisfies explicit   Markov  assumptions; See e.g., \cite[Proposition 8]{Schweiger:00}
and \cite[Prop. 2.12]{broise}. The Markov  partition is  given by the  following set  indexed by permutations $\sigma$ on  $m$ elements:
$$P_{\sigma}=\{  (x_1, \ldots, x_m)\in [0,1]^m :0 \leq x_{\sigma(1)} \leq \ldots \leq x_{\sigma(m)} \leq 1 \}.$$ 
%$S_m$ stands for the set of permutations on  $m$ elements.

We give an explicit description of the above facts in the case of $m=2$.
The map  $T_{\rm JP}$  satisfies $$
T_{\rm JP}(\xi,\eta) =
\begin{cases}
(\{\frac{\eta}{\xi} \},\{ \frac{1}{\xi} \})& \hbox{ if } \xi \neq 0\\
(0, \eta) &  \hbox{ if } \xi=0 .
\end{cases}
$$
Denote by $I_{a,b}$, for $(a, b) \in \mathbb{Z}^2$ with $0 \leq a \leq b$ and $b \geq 1$, the sets 
\[ I_{a,b}=
\begin{cases}
\{ (\xi, \eta) \in (0,1)^2: 1/(b+1) < \xi < 1/b , \  a \xi < \eta < (a+1)\xi  \} &  \hbox{ if } a \neq b \\
\{ (\xi, \eta) \in (0,1)^2: 1/(b+1) < \xi < 1/b , \  a \xi < \eta < 1  \}  &  \hbox{ if } a = b,
\end{cases}
\]
 which form a disjoint partition for $[0,1]^2$. The map $T_{\rm JP}$ is not a full branch map and $\mathcal P$ consists of two elements   $P_1= \{(\xi,\eta ) \in [0,1]^2: \xi<\eta\}$ and $P_2= \{(\xi,\eta ) \in [0,1]^2: \xi > \eta\}$. More precisely, we have 
$$
T_{\rm JP}(I_{a,b}) =
\begin{cases}
(0,1)^2 & \hbox{ if } a \neq b\\
P_1 &  \hbox{ if } a=b.
\end{cases}
$$

Let $h_{a,b}:  T_{\rm JP}(I_{a,b}) \to I_{a,b}$ be the inverse branch associated with  $I_{a,b}$. For $(\xi, \eta) \in (0,1)^2 \cap \mathbb{Q}^2$, we have an admissible sequence $(a_1, b_1), \ldots , (a_n, b_n)$ with $b_n \geq 2$ such that for all $i \in \{1, \ldots, n-1\}$
\[ T_{\rm JP}^{i-1}(\xi, \eta) \in I_{a_i, b_i}, \ T_{\rm JP}^{n-1}(\xi, \eta) \in \overline{I_{a_n, b_n}},   \mbox{ and }
T_{\rm JP}^n(\xi, \eta)=(0, 0).\] This yields 
\[ (\xi, \eta)= h_{a_1, b_1} \circ \ldots \circ h_{a_n, b_n} (0, 0).
\]
In the particular case that $x_0=(0,0)$, we thus can identify $X = \mathbb Q^2$ and the image of $x_0$ under the elements in $\mathcal H^n$ corresponds to a pair of rationals $(\frac{p}{q}, \frac{r}{q})$ with $(p,q,r)=1$. Moreover, the associated weight becomes $w(p/q, r/q)=3 \log q$.

Next, the non-arithmeticity is proved by comparing two  denominators for   convergents  of   periodic orbits as in  the  previous cases.
Indeed, by \cite{PLRD},  in every real number field  ${\mathbb K}$ of degree $m+1$, there exists  $(x_1, \ldots,x_m) \in (0,1)^m$
having a purely periodic  Jacobi--Perron  expansion  such that $(1,x_1,\ldots,x_m)$ is a basis of ${\mathbb K}$.
Consider two   distinct real number field  ${\mathbb K}$  and  ${\mathbb K}'$, and $(x_1, \ldots,x_m) $, $(x'_1, \ldots,x'_m)$  having 
purely periodic  Jacobi--Perron  expansions,  with  $(1,x_1,\ldots,x_m)$ being  a basis of ${\mathbb K}$ and $(1,x'_1,\ldots,x'_m)$ being  a basis of ${\mathbb K}'$.
The  Jacobians  are  given by the common  denominator  of the   homographies involved  in the inverse branches.  The denominators  of the  homographies in ${\mathcal H}^*$ are of the form 
$ q_{n-m}  x_1  + \ldots  + q_{n-1} x_m +q_n$  with $(q_n)_n$  taking positive  integer  values,
and similarly  for   $(1,x'_1,\ldots,x'_m)$,  the denominators  of the  homographies   are  of the form 
$ q'_{n'-m}  x'_1  + \ldots  + q'_{n'-1} x'_m +q'_{n'}$  with $(q'_n)_n$  taking integer  values.
One sees that  
\[
\frac{\log  ( q'_{n'-m}  x'_1  + \ldots +q'_{n'-1} x'_m +q'_{n'})}{ \log  ( q_{n-m}  x_1  + \ldots q_{n-1} x_m +q_n)} \not \in {\mathbb Q}
\]
since  $(1,x_1,\ldots,x_m)$ and $(1,x_1,\ldots,x_m)$ belong to the distinct field extensions.
\end{proof}

\subsection{Remarks} \label{subsec:rem}

Though we have restricted our focus to the generalised continued fraction maps in higher dimensions, the same ideas might apply to a wide family of instances such as suitable accelerations of the Rauzy--Veech map  in the setting of interval exchanges, or induced Bowen--Series maps (e.g., Romik's map \cite{Romik}, Rosen's map \cite{Rosen}). In these examples, Assumption \ref{assmp:exp} is known or easily verifiable, however, Assumption \ref{assump:nonarith} has to be checked case by case as our space $X$ and periodic orbits are not completely characterisable (cf. \cite{Sch} for ``Rosen's cusp challenge"), hence the above tricks might not work for some cases.

We stress the fact there are some algorithms for which the current method would not apply. For example, Selmer's algorithm does not admit natural accelerations and the invariant density is unbounded  by \cite[Chapter 7]{Schweiger:00}, while our transfer operators need the density  eigenfunction to lie in the  set of Lipschitz bounded functions. We refer to Cantrell--Pollicott \cite{CanPol} and \cite{CanPol2} for other geometric contexts where the Dolgopyat estimate is not accessible. 

As a final remark, our motivation also comes from the applications to number theory. We recall the recent progress by Bettin--Drappeau \cite{BetDra} and Lee--Sun \cite{LeeSun} on Mazur--Rubin's conjectural statistics on modular symbols and $L$-functions of $\GL(2)$, where they presented dynamical proofs by generalising \cite{baladi-vallee}. In particular, the main technical difficulties heavily depend on verifying the Dolgopyat estimates and topological mixing property of a variant of the Gauss map. Our result overcomes this point in a precise way, hence would yield simpler proofs. We expect further extensions in this direction.

%%%%%%%%%%%%%%%%%%%%%%%%%%%%%%%%%%%%%%%%%%%%%%%%%%%%%%%%%%%%%%%%%%%%%%%%%%

\section{Spectral properties of transfer operators} \label{sec:dynamics}

We first derive from  our assumptions  basic  spectral properties of  the  transfer operators associated with the dynamical   system $(I,T)$.

Let $\mathrm{Lip}(P_a)$ be the space of functions $f:P_a \rightarrow \mathbb{C}$ which can be extended to Lipschitz functions on the compact closure $\overline{P_a}$. Any $f \in \mathrm{Lip}(P_a)$ has a Lipschitz constant 
\[ \mathrm{Lip}(f|_{P_a})=\sup_{x \neq y \in P_a} \frac{|f(x)-f(y)|}{|x-y|} . \]

%\tcb{Lipschitz Assumption on $I_j$ or on $P$?}
\begin{definition} \label{def:lip}
We define  the space of functions 
$B :=\oplus_{a \in A} C^{1+\mathrm{Lip}}(P_a)$
%= \{ f: I \rightarrow \mathbb{C}: f|_{P_a} \in \mathrm{Lip}(P_a) , a \in A \} \]
with the norm 
\[ \|f\| := \| f \|_\infty + \max_{a \in A} \mathrm{Lip}(f|_{P_a})  \]
where $ \| f \|_\infty=\sup_{x\in I} |f(x)|$.
\end{definition}

We note that $B$  is a Banach space because $\mathrm{Lip}( P_a)$ is a Banach space for any $a \in A$ using the fact that any function in $\mathrm{Lip}(P_a)$ has a unique extension to the closure $\overline{P_a}$ with the same Lipschitz constant. Further, the norm $\| \cdot \|$ is pre-compact for the topology of the norm $\| \cdot \|_\infty$ by the Arzela--Ascoli Theorem.

To prove Proposition \ref{prop.analyticity} we introduce families of transfer operators. Let $s \in \mathbb{C}$ and $t \in \mathbb{C}^d$. To proceed, we define an extended notion of \eqref{eq:weight} as follows. For $h \in \mathcal{H}^\ast$ and $y \in \mathrm{Im}(h)$, set
\[ w(y):=-\log|J_h(x)| , \quad y=h(x), x \in \mathrm{Dom}(h)  \]
by abusing the notation, as for all $y \in X$ we always have $x=x_0$.

%{\blue [Back to previous notations for Def 2.2 and 2.3. Seemed confused without $w(y)$. Also for the use of convexity argument.] }
\begin{definition}
Let $\widetilde{\mathcal{L}}_{s,t} : B \to B$ be the transfer  operator defined by
\[
\widetilde{\mathcal{L}}_{s,t}f(x) = \sum_{y : y = h(x) \atop h \in \mathcal{H}, x \in \mathrm{Dom}(h)} e^{s w(y)+ \langle t, \overline N(y) \rangle} f(y).
\]
We denote by $\widetilde{\mathcal{L}}_{s,t}^\sharp$ the same operator defined by the inverse branches corresponding to $\mathcal{H}^\sharp \subset \mathcal{H}$ for the unique representation for $x \in X$ in case that there is a duplication. 
\end{definition}

\begin{remark}
Recall Example \ref{ex:intro:gauss} that for the Gauss map, it is given by 
\[ \mathcal{H}^\sharp = \{ h_j : j \geq 2 \} \]
which corresponds to the unique terminating condition for finite continued fraction expansions of the rational. In other cases, it is possible to have $\mathcal{H}^\sharp=\mathcal{H}$.
\end{remark}

It will be convenient to normalise the transfer operator. We use the centred  counting function $\overline{\varphi }$ defined in (\ref{eq:phibar}), that is, $\overline{\varphi}(x) = \overline{N}(x) - w(x)\overline{\Lambda}$, where $\overline{\Lambda} = (\Lambda_{j_1}, \ldots, \Lambda_{j_d}) \in \R^d.$
\begin{definition}\label{def:normalised}
We define  the normalised transfer operator $\mathcal{L}_{s,t} : B \to B$ as  the operator defined by
\[
\mathcal{L}_{s,t}f(x) = \sum_{y : y = h(x) \atop h \in \mathcal{H}, x \in \mathrm{Dom}(h)} e^{s w(y)+ \langle t, \overline{\varphi}(y) \rangle} f(y).
\]
Similarly we denote by $\mathcal{L}_{s,t}^\sharp$ the final operator corresponding to $\mathcal{H}^\sharp$.
\end{definition}

By the choice of our weight $w(\cdot)$ given by the Jacobian determinant, one can easily see that our transfer operators converge for $\mathfrak{Re}(s)$ and $\mathfrak{Im}(t)$ belonging to a real neighborhood  of $(1,0)$
by Assumption \ref{assmp:exp}.(1) together with the use of the mean value theorem.

\begin{remark} \label{rmk:cell}
In \cite{kim2025}, they identified $\mathcal{P}$ as a CW-complex as follows. For $0 \leq i \leq m=\mathrm{dim}(I)$, let $\mathcal{P}[i]$ be the set of open cells of real dimension $i$. Then one has a finite structure $\mathcal{P}=\cup_{i=0}^m \mathcal{P}[i]$ that gives a  decomposition of the function space $B=\oplus_{i=0}^m B(\mathcal{P}[i])$ equipped with  reasonable norms and of the operator $\mathcal{L}:=\mathcal{L}_{s,t}$ as a lower-triangular matrix 
\[  \Lc = \begin{bmatrix}
\Lc_m^m & 0 & \cdots & 0  \\
\Lc_m^{m-1} & \Lc_{m-1}^{m-1} & \cdots &0  \\
\vdots & \vdots & \vdots & \vdots  \\
\Lc_m^0 & \Lc_{m-1}^0 & \cdots  & \Lc_0^0 
\end{bmatrix} \]
where $\mathcal{L}^i_k: B(\mathcal{P}[k]) \rightarrow B(\mathcal{P}[i])$ with $0 \leq i,k \leq m$ is the component operator.  This allows the boundaries of $P$'s, $P \in \mathcal{P}$ when $|\mathcal{P}|>1$, to be taken into account, which means the elements in $B$ are honest functions rather than equivalence classes of them. 

For complex continued fraction maps ($m=2$) and transfer operators acting on $C^1$-functions, they obtained spectral properties that enable to deduce the counting statistics without multiplicity. We speculate that a similar argument would apply to our setting for general expanding maps on $\R^m$ and analysis on Lipschitz functions. In our setting, we can track the multiplicities thanks to Assumption \ref{assmp:exp}.(3); $I_j \subset P_a$, so we prefer not to try to work in greater generality.
 \end{remark}

To  establish  the  required properties on the spectrum of $\widetilde{\mathcal{L}}_{s,t}$,  we first state the following result on the  density of backward orbits.

\begin{lemma} \label{lem:density}
For any $x_0 \in I$ we have that
\[
\overline{\bigcup_{n \ge 1} \mathcal{H}^n(x_0)} = I,
\]
that is, the collection of pre-images of $x_0$ under $T$ is dense in $I$.
\end{lemma}
\begin{proof}
Fix $\epsilon >0$ and $x \in I$. Now take $n$ sufficiently large so that $x_0 \in I_{j_1,\ldots, j_n}, x \in I_{k_1,\ldots,k_n}$
and $|I_{k_1,\ldots,k_n}|, |I_{j_1,\ldots,j_n}| \le \epsilon/2$. Then by the  transitivity assumption from Assumption \ref{assmp:exp}, there exists $N \ge 1$ such that $\textnormal{int}(I_{k_1,\ldots,k_n}) \cap T^N\textnormal{int}(I_{j_1,\ldots,j_n}) \neq \emptyset$ and so there exist $h \in \mathcal{H}^N, y \in I$ so that $y \in \textnormal{int}( I_{j_1,\ldots,j_n}) \cap h(\textnormal{int}(I_{k_1,\ldots,k_n}))$. Hence 
\[
|h(x) - x_0| \le |h(x) - y| + |y - x_0| \le \epsilon
\]
and we are done.
\end{proof}

Our first result on the positive real transfer operator is then as follows. Let $K$ be a real neighbourhood of 1, or the region $\sigma>1-\epsilon$ for some $\epsilon>0$, where $\sum_{h \in \mathcal{H}} \sup |J_h|^\sigma$ converges.

\begin{proposition} \label{thm:PF}
For each $\sigma \in K$ the operator $\widetilde{\mathcal{L}}_{\sigma,0}$ has finite spectral radius $R_\sigma > 0$. Furthermore, $\widetilde{\mathcal{L}}_{\sigma,0}$ has a simple, maximal eigenvalue $\widetilde{\lambda}_{\sigma}$ with a corresponding strictly positive eigenfunction $h_\sigma$. All eigenvalues of modulus $R_{\sigma}$ are simple and all other eigenvalues lie in a disk $\{z \in \mathbb{C} : |z| < \theta_{\sigma} R_{\sigma}\}$ for some $0 < \theta_{\sigma} < 1$. 

In particular, $\widetilde{\mathcal{L}}_{1,0}$  has a strictly positive eigenfunction $h \in B$ associated to the eigenvalue~$1$.
\end{proposition}
\begin{proof}
The proof follows the proof \cite[Theorem 1.5]{Baladi}. We explain how the assumptions on our dynamical systems are used in the proof.

For the existence of $\widetilde{\lambda}_{\sigma}$, we show the operator satisfies the Lasota--Yorke inequality. For all $n \geq 1$, $a \in A$ and $x,y \in P_a$, one has 
\begin{align*}
|\widetilde{\mathcal{L}}_{\sigma, 0}^n f(x) - \widetilde{\mathcal{L}}_{\sigma,0}^n f(y)| & \leq \sum_{h \in \mathcal{H}^n \atop x,y \in \mathrm{Dom}(h)} (|J_h(x)-J_h(y)|^\sigma | f \circ h(x)| + |J_h(y)|^\sigma |f \circ J_h(x)-f \circ J_h(y)|) \\
& \leq \sum_{h \in \mathcal{H}^n \atop x,y \in \mathrm{Dom}(h)} |J_h(y)|^\sigma \left| \frac{J_h(x)}{J_h(y)}-1\right| |f \circ h(x)| + \sum_{h \in \mathcal{H}^n \atop x,y \in \mathrm{Dom}(h)} |J_h(y)|^\sigma \|f \| |h(x)-h(y)| \\
&\ll_K  M \|f \|_\infty + \rho^n  \| f \| 
\end{align*}
by Assumption \ref{assmp:exp}, where the implicit constants depend only on $(I,T)$. 

By Hennion's criterion \cite[Theorem XIV.3]{hennion}, the operator $\widetilde{\mathcal{L}}_{\sigma, 0}$ is then quasi-compact with the maximal eigenvalue $\widetilde{\lambda}_{\sigma}>0$. All other eigenvalues lie in a disk $\{z \in \mathbb{C} : |z| < \theta_{\sigma} R_{\sigma}\}$ for some $0 < \theta_{\sigma} < 1$. We also have $R_\sigma \leq \widetilde{\lambda}_\sigma=\lim_{n \rightarrow \infty} \| \widetilde{\mathcal{L}}_{\sigma, 0}^n 1 \|_\infty^{1/n}$ and conversely
\[ R_\sigma \geq \lim_{n \rightarrow \infty} \| \widetilde{\mathcal{L}}_{\sigma, 0}^n 1 \|^{1/n} \geq \lim_{n \rightarrow \infty} \| \widetilde{\mathcal{L}}_{\sigma, 0}^n 1 \|_\infty^{1/n}=\widetilde{\lambda}_\sigma>0. \]

Suppose now that $h_\sigma(x)=0$ for some $x \in X$. Then for any $n \geq 1$
\[ 0=\widetilde{\mathcal L}_{\sigma,0} ^n h_\sigma(x)= \sum_{y=h(x) \atop h \in \mathcal{H}^n, x \in \mathrm{Dom}(h)} e^{-\sigma w(y)} h_\sigma(y) . \]
Since $e^{-\sigma w(y)}>0$ for all $y$ and $h_\sigma$ is continuous, the eigenfunction $h_\sigma$ should be identically zero, which contradicts Lemma \ref{lem:density}. Similarly the eigenvalues of modulus $R_{\sigma}$ are simple again due to Lemma \ref{lem:density}, as one can deduce that if $h_{\sigma,1}$ and $h_{\sigma,2}$ are two eigenfunctions for $R_\sigma$ then they are $\R$-linearly dependent.

For the last item, we observe that $\widetilde{\mathcal L}_{1,0}$ is simply weighted by $1/|J_h|$ hence the change of variable directly shows that $1$ is the dominant eigenvalue with a strictly positive eigenfunction $h \in B$.
\end{proof}

\textbf{Normalisation assumption:} Moving forward we will assume, without loss of generality, that  $R_1 = 1$.

\begin{lemma}[Spectral properties]\label{lem.spectrum}
The operator $\widetilde{\mathcal{L}}_{s,t}$ satisfies the following: 
\begin{enumerate}[(i)]
\item
For $(s,t)$ in a neighbourhood of $(1,0)$, the operators $\widetilde{\mathcal{L}}_{s,t}$ have a simple maximal eigenvalue $\widetilde{\lambda}(s,t)$
 that varies analytically and the rest of the spectrum is contained in a disk of strictly smaller radius that $|\widetilde{\lambda}(s,t)|$ uniformly in $s,t$.
\item
For any fixed $s \neq 1 $ with $\mathfrak{R}(s) \ge 1$, there exists $\epsilon(s)>0$ such that if $|t|< \epsilon(s)$, then $\widetilde{\mathcal{L}}_{s,t}$ has spectral radius at most $1$ and does not have $1$ as an eigenvalue. 
\end{enumerate}
\end{lemma}

\begin{proof}
Part (i)  is now a consequence of Proposition \ref{thm:PF} and analytic perturbation theory (See e.g., \cite{kato}).

For part (ii), we claim that for $\xi \neq 0$ we have that $1$ is not in the spectrum of $\widetilde{\mathcal L}_{1+i\xi,0}$.  Assume for a contradiction that $1$ is in the spectrum.  Then by the quasi-compactness from (i), it would correspond to an eigenvalue $\widetilde{\mathcal L}_{1+i\xi,0} g = g$.  Since $\widetilde{\mathcal L}_{1,0} 1 = 1$ (this is by the Normalisation assumption and by replacing $g$ by $w+\log g \circ T -\log g$) and $|\widetilde{\mathcal L}_{1+i\xi,0} g(x)| = |g(x)|$, a simple convexity argument  (See e.g., \cite[Chapter 4]{ParryPollicott}) shows that $|g|=1$ and whence $w$ is cohomologous to a function taking values in $c\Z$ for some $c \in \R$. However this gives a contradiction by evaluating on periodic points in view of the non-arithmeticity  Assumption  \ref{assump:nonarith}. 
\end{proof}

\begin{remark} \label{rem:acim}
Theorem \ref{lem.spectrum}.(iii)  yields the existence of an  ergodic absolutely continuous  invariant  measure for the  map $T$.
 It thus makes sense to compare  the behaviour of  finite orbits to the behaviour of generic  truncated orbits.
\end{remark}

%\[
%\left. \frac{\partial^n}{\partial t_1^{n_1} \partial t_2^{n_2} \cdots \, \partial t_d^{n_d}} \right|_{(s,0)}  \frac{F(s,t)}{1-\lambda(s,t)}
%\]

%\begin{remark}\label{rem:deriv}
%{ \blue [This is explained in the proof of Lem 2.10. in the right below] }
%The notation $\widetilde{\lambda}_s(s,t)$  (respectively $\widetilde{\lambda}_t(s,t)$) stands for  the partial  derivative  with respect to the  first  (respectively second) coordinate of  the two-variable function $\widetilde{\lambda}(s,t)$.
%By differentiating with respect to $s$  the eigenvalue equation $ \widetilde{\mathcal{L}}_{s,0} g_{s,0} =  \widetilde{\lambda}(s,0) g_{s,0} $ (where $g_{s,t}$ is the eigenfunction of 
%$\widetilde{\mathcal{L}}_{s,t}$ associated  with $ \widetilde{\lambda}(s,t)$),  and then integrating with respect to the eigenmeasure  we see that $\widetilde{\lambda}_s(1,0) < 0$. 
 % Indeed 
% $$ \widetilde{\mathcal{L}}_{s,t}  (-w g_{s,0})+   \widetilde{\mathcal{L}}_{s,t}  g'_{s,t}=    \widetilde{\lambda}_s (s,0)  g_{s,0}+  \widetilde{\lambda}(s,0) g'_{s,t}$$
% and  we use $$\int \widetilde{\mathcal{L}}_{s,t}  g'_{s,t}= \int  g'_{s,t}.$$

%We will see in the proof of Theorem \ref{thm.firstmoment} that the constant $\Lambda_j$ is given by
%\begin{equation}\label{eq:lambdaq}
%\Lambda_j = - \frac{\widetilde{\lambda}_t(1,0)}{\widetilde{\lambda}_s(1,0)}.
%\end{equation}
%\end{remark}

\begin{remark}
It is clear that the normalised transfer operator $\mathcal{L}_{s,t}$ admits the same spectral properties as $\widetilde{\mathcal{L}}_{s,t}$ as described in Lemma \ref{lem.spectrum}. We further have 
 ${\lambda}(s,t)= \widetilde{\lambda}(s- \langle  \overline{\Lambda}, t \rangle,t)$ for $(s,t)$ in a neighbourhood of $(1,0)$.
\end{remark}
 It is then easy to show the following. %\tcb{To prove it later?}
\begin{lemma}\label{lem.deriv}
We have that $\lambda_s(1,0) < 0$, and writing $t = (t_1,\ldots, t_d)$, we have that
$\lambda_{t_i}(1,0) = 0$ for each $i=1,\ldots, d$.
%$ e^{P_t(1,0,0)} = e^{P_u(1,0,0)} =1$
Furthermore the matrix $(\sigma_{i,k})$ where
\[
\sigma_{i,k} = -\frac{\lambda_{t_i,t_k}(1,0)}{\lambda_s(1,0)} 
\]
 for $i,k \in \{1,\ldots,d\}$,
is positive definite.
\end{lemma}

\begin{proof}
Since the operator
$\mathcal L_{1,0} : B \to B$ has isolated  simple maximal eigenvalue $1$ the associated eigenvalue $\lambda_{s,t}$
and associated eigenfunction $h_{s,t}$
 for   $\mathcal L_{s,t} : B \to B$ have an analytic dependence $(s,t)$ for $|s-1|, \|t\|$ sufficiently small.
 
 We can differentiate the eigenfunction $\mathcal L_{s,0} g_{s,0} = \lambda_{s,0} g_{s,0}$  in $s$ to get
$$\mathcal L_{s,0}( -w g_{s,0}) + \mathcal L_{s,0}( \partial_sg_{s,0}) =  \partial_s \lambda_{s,0} g_{s,0}
+ \lambda_{s,0}   \partial_s g_{s,0}.
$$
We can set $s=0$ and apply the dual eigenvalue equation $\mathcal L_{1,0}^* \mu = \lambda_{1,0} \mu$
and cancelling terms
to deduce that 
$$
\lambda_s(1,0):= \partial_{s=1} \lambda_{s,0} =  \mu \left( \sum_{j \in J} \int_{I_j} -\log|J_{h_j}(x)| g_{1,0}(x) \ dx \right) < 0.
$$
A similar argument differentiating in the variables $t_i$ shows that $\lambda_{t_i}(1,0):=\partial_{t_i=0} \lambda_{1,t} = \mu (\varphi_i) = 0$ for each $i$.

The expressions for the second derivatives start from the eigenvalue identities $\mathcal L_{1,t}^n g_{1,t} = \lambda_{1,t}^n g_{1,t}$
(for $n \geq 1$).  Differentiating once in $t_i$ (at $s=0$) gives 
$$
\mathcal L_{1,t}^n( (S_n\varphi_i) h_{1,t}) + \mathcal L_{1,t}^n( \partial_th_{1,t}) =  
n  \lambda_{1,t,0}^{n-1} \partial_t \lambda_{1,t,0} h_{1,t,0}
+ \lambda_{1,t,0}^n   \partial_t h_{1,t,0}.
$$
where $S_n\varphi_i = \varphi_i + \varphi_i\circ T + \cdots + \varphi_i\circ T^{n-1}$
and  differentiating again  in $t_i$ gives 
$$
\begin{aligned}
&\mathcal L_{1,t}^n( (S_n\varphi_i)^2 h_{1,t})  
%+\mathcal L_{1,t,0}^n( w  \partial_th_{1,t,0}) 
+ 2 \mathcal L_{1,t}^n( (S_n\varphi_i)
\partial_t
g_{1,t}) + \mathcal L_{1,t}^n( \partial_t^2g_{1,t})
\cr
&=  
n(n-1)  \lambda_{1,t}^{n-2} (\partial_t \lambda_{1,t})^2 g_{1,t}
+  n  \lambda_{1,t}^{n-1} \partial_t^2 \lambda_{1,t} g_{1,t}
+
2  n  \lambda_{1,t}^{n-1} \partial_t \lambda_{1,t} \partial_t g_{1,t}
%+ \lambda_{1,t,0}   \partial_t h_{1,t,0}
+ \lambda_{1,t}^n   \partial_t^2 g_{1,t}.
\end{aligned}
$$

Setting $t=0$, applying the dual eigenvalue equation  and cancelling terms, dividing by $n$ and letting $n \to \infty$ we have the expression
$$
\lambda_{t_i, t_i}(1,0) = \lim_{n \to \infty} \frac{1}{n}\mu\left( (S_n\varphi_i)^2 \right) 
$$ 
at $t=0$.
A similar argument gives
$$
\lambda_{t_i,t_j}(1,0)  =  \lim_{n \to +\infty} \frac{1}{n}\mu\left((S_n\varphi_i) (S_n\varphi_j) \right).
$$ 
at $t=0$.
To show that the matrix $A = (\lambda_{t_i,t_j}(1,0))_{i,j}$
is positive define we can consider any vector $v =(v_1,\ldots, v_d)  \in \mathbb R^d \backslash \{0\}$ and then we can write 
$$
\begin{aligned}
v A v^T &= %\lambda_{tt} x^2 + 2  \lambda_{tu} x y + \lambda_{uu} y^2\cr 
%&= x^2 \lim_{n \to +\infty} \frac{1}{n}\mu\left((S_n\phi_1)^2 \right)
%+ 2 x y \lim_{n \to +\infty} \frac{1}{n}\mu\left((S_n\phi_1) (S_n\phi_2) \right)
%+ y^2 \lim_{n \to +\infty} \frac{1}{n}\mu\left((S_n\phi_2)^2 \right)\cr
%&=
\lim_{n \to \infty} \frac{1}{n}\mu\left(\sum_{i=1}^d v_i S_n \varphi_i \right)^2 \geq 0.
\end{aligned}
$$
\end{proof}

We would like to know when the matrix $(\sigma_{i,k})$ is strictly positive definite.

\begin{proposition}\label{prop.var}
Suppose that $d=1$, i.e., we are working with a single cost function. Then $\lambda_{t,t}(1,0) > 0$.
\end{proposition}

\begin{proof}
Following a well-known argument for transfer operators, we see that $\lambda_{t,t}(1,0) = 0$ if and only if when $x \in I$ satisfies that $h(x) =x$ for $h = h_1 \circ \cdots \circ h_n \in \mathcal{H}^n$ then
\[
N_j(x) - \Lambda_j w(x) = \sum_{0 \le i \le n-1} c_j(T^i(x)) + \Lambda_j \log|J_h(x)|
\]
is equal to $C n$. However, it is easy to see that we must have $C =0$ by our choice of $\Lambda_j$. This would contradict Assumption \ref{assump:nonarith} as $\sum_{0 \le i \le n-1} c_j(T^i(x)) $ takes values in $\Z$.
\end{proof}

It is easy to generalise this to the following multidimensional version.

\begin{proposition}\label{prop.multivar}
The matrix $(\sigma_{i,k})$ from Lemma \ref{lem.deriv} is strictly positive definite if and only if there does not exist $v \in \R^d$ such that 
\[
 \left\langle v , \overline{\varphi}(x) - w(x) \overline{\Lambda}  \right\rangle = 0
 \]
 for all  $x \in I$ such that $h(x) = x$ for some $h \in\mathcal{H}$.

\end{proposition}

\section{Proofs of the main results}

\subsection{Law of large numbers}\label{subsec:LL}

%Recall that $\Lambda_p = -\widetilde{\lambda}_t(1,0)/\widetilde{\lambda}_s(1,0)$ in the following.

The proof of Theorem \ref{thm.firstmoment}   relies on the  following classical result due to Delange \cite{delange}  which will allow us to convert results on complex functions to asymptotic formulae. 

\begin{proposition}[Delange Tauberian Theorem] \label{prop.tau}
For a monotone increasing function $\phi : \mathbb{R}_{>0} \to \mathbb{R}_{>0}$, we set
\[
f(s) = \int_0^{\infty} e^{-sQ} \ d\phi(Q).
\]
Suppose that there is $\delta > 0$ such that
\begin{enumerate}
\item $f(s)$ is analytic on $\{\mathfrak{Re}(s) \ge \delta\} \backslash \{\delta\}$; and,
\item  there are positive integers $n, k  \ge 1$,  $ \delta$ a real  number, an open neighbourhood $U$ of  $\delta$, non-integer numbers $0 < \mu_1, \ldots, \mu_k < n$ and analytic maps $g,h, l_1, \ldots, l_k : U \to \mathbb{C}$ such that
\[
f(s) = \frac{g(s)}{(s-\delta)^n} + \sum_{j=1}^k \frac{l_j(s)}{(s-\delta)^{\mu_j}}+ h(s) \ \text{ for $s \in U$ and  $g(\delta) > 0$}.
\]
\end{enumerate}
Then
\[
\phi(Q) \sim \frac{g(\delta)}{(n-1)!} Q^{n-1} e^{\delta Q}
\]
as $Q\to\infty$.
\end{proposition}

We will use this result and follow the main arguments present in \cite{CanPol2}.

\begin{proof}[Proof of Theorem \ref{thm.firstmoment}]
Let $X$  be defined as in (\ref{eq:defX}). Fix $
j \in J$.
%It suffices to work with  the one-dimensional map $x \mapsto c_j(x)$.
We define the two-variable series
\[
\zeta(s,t) = \sum_{x \in X} e^{-sw(x) + t  N_j(x)}
\]
for $s,t\in \C$.
Note that
\[
\zeta(s,t) = \widetilde{\mathcal{L}}_{s,t}^\sharp \circ \sum_{n\ge 0} \widetilde{\mathcal{L}}_{s ,t}^n \1(x_0)
\]
where $ \1$ stands for the  function taking constant value $1$ in $B$.

It follows from Lemma \ref{lem.spectrum} that for any $s_0 \in \C$ with  $\mathfrak{Re}(s_0) \ge 1$,  there exist $\epsilon(s_0), \delta(s_0) > 0$ such that $\zeta(s,t)$ is bi-analytic  for  the set of  $(s,t)$ satisfying  $|s-s_0|\le \delta(s_0)$ and $|t| \le \epsilon(s_0)$.  For $(s,t)$ close to $(1,0)$ we can write 
\[
\widetilde{\mathcal{L}}_{s,t}^n\textbf{1}(x_0) = \widetilde{\lambda}(s,t)^n \widetilde{Q}(s,t)\textbf{1}(x_0) + O(\theta^n)
\]
where $0 < \theta < 1$ is independent of $s,t$.
 It follows that
\begin{equation} \label{eq.double}
\zeta(s,t) = \frac{ \widetilde{\lambda}(s ,t)\widetilde{Q}(s,t)\textbf{1}(x_0)}{1-\widetilde{\lambda}(s,t)} + R(s,t)
\end{equation}
where $R(s,t)$ is bi-analytic in a neighbourhood of $(1,0)$. Set $F(s,t) = \widetilde{\lambda}(s,t)\widetilde{Q}(s,t)\textbf{1}(x_0)$ and note that $F(1,0) > 0$ (as $\widetilde{\mathcal{L}}_{1,0}$ has strictly positive eigenfunction associated to the eigenvalue $1$). 

When $t =0$, we deduce that 
\[
\zeta(s,0) = \frac{  - \widetilde{\lambda}_s(1,0)^{-1} F(1,0)}{s-1} + \tilde{R}(s)
\]
for $\mathfrak{Re}(s) \ge 1$, where $\tilde{R}$ is analytic on this domain. It follows from Proposition \ref{prop.tau} (here $\delta=1$)  together with Lemma \ref{lem.deriv} that
\[
\#\{ x \in X : w(x) < Q\} \sim  -\widetilde{\lambda}_s(1,0)^{-1} F(1,0) \, e^{Q}
\]
as $Q \to \infty$.

We now differentiate $(\ref{eq.double})$ with respect to $t$ at $t=0$. This yields 
\[
\sum_{n \ge 1} N_j(x) e^{-sw(x)} = \zeta_t(s,0) = \frac{ - \widetilde{\lambda}_t(s,0)\widetilde{\lambda}(s,0)\widetilde{Q}(s,0)\1(x_0)}{(1-\lambda(s,0))^2} + L(s)
\]
where $L$ is analytic on $\mathfrak{R}(s)\ge 1$ apart from a possible simple pole at $s=1$. It follows easily that $\zeta_s(s,0)$ has a pole of residue $\widetilde{\lambda}_s(1,0)^{-2} \widetilde{\lambda}_t(1,0) \widetilde{\lambda}(1,0)\widetilde{Q}(1,0)\1(x_0)$ at $s=1$. Applying Proposition \ref{prop.tau} gives that
\[
\sum_{w(x) <Q} N_j(x) \sim \widetilde{\lambda}_s(1,0)^{-2} \widetilde{\lambda}_t(1,0) F(1,0) Q e^Q
\]
as $G\to\infty$. Combining this with our asymptotic above shows that

\[
\lim_{Q\to\infty} \frac{1}{\#\{ x \in X : w(x) < T\}}  \sum_{w(x) < Q} \frac{N_j(x)}{Q} = - \frac{\widetilde{\lambda}_t(1,0)}{\widetilde{\lambda}_s(1,0)} = \Lambda_j 
\]
 as required.
\end{proof}

\subsection{Large deviations}\label{subsec:LD}

We will deduce our large deviation theorem  (Theorem \ref{thm.largedev})   from   the following local G\"artner-Ellis type Theorem.
%\begin{definition}
%Let $\beta_n$ be a sequence tending to $\infty$ and $\xi$ a non-zero real number. A sequence of random variables $X_n$ with $\E(X_n) \sim \xi \beta_n$ satisfies a large deviation property relative to the interval $[x_0,x_1]$ containing $\xi$ if a function $W(x)$ exists, such that $W(x) > 0$ for $x \neq \xi$
%\[
%\frac{1}{\beta_n} \log \mathbb{P}(X_n \le x \beta_n) = -W(x) + o(1) \ \ \text{ when $x_0 \le x \le \xi$}
%\]
%and 
%\[
%\frac{1}{\beta_n} \log \mathbb{P}(X_n \ge x \beta_n) = -W(x) + o(1) \ \ \text{ when $\xi \le x \le x_1$}
%\]
%as $n\to\infty$.
%\end{definition}
%\begin{theorem}[Quasi-powers for large deviations]
%Suppose that a sequence of random variables $X_n$ has the property that the moment generating functions $\lambda_n(s) = \E(e^{sX_n})$ are analytic in a real interval $[-\rho, \rho]$ for some $\rho >0$. Further assume that
%\[
%\lambda_n(s) = e^{\beta_n  U(s) + V(s)}\left( 1 + O\left(\frac{1}{k_n} \right)\right)
%\]
%uniformly in $s$ as $n \to\infty$. Then there exist $x_0 < 0 < x_1$ such that $X_n$ satisfies a large deviation principle relative to the interval $[x_0,x_1]$.
%\end{theorem}
\begin{lemma}\label{lem.localldp} [Lemma 3.11 in \cite{CanPol2}]
Let $Z_n$ be a sequence of real random variables and $\mu_n$ be a sequence of probability measures on a space $X$.  Suppose that
there exists $\eta>0$ with
\[
\lim_{n\to\infty} \frac{1}{n}\log \mathbb{E}_n(e^{tZ_n}) = c(t)
\]
for all $t \in [-\eta, \eta]$ where $\E_n$ represents the expectation with respect to $\mu_n$.
If $c$ is continuously differentiable and strictly convex on $[-\eta,\eta]$ and $c^\prime(0)=0$, then, for $0<\epsilon< \eta$,
\[
\limsup_{n\to\infty} \frac{1}{n}\log \mu_n(Z_n>n\epsilon) \le -\sup_{0<t<\eta} \{t\epsilon - c(t)\} <0.
\]

\end{lemma}

%\begin{proof}
%We note that for $0 < \epsilon < \eta$ and for any $t > 0$, we have that 
%\[
%\mu_n(Z_n > n \epsilon) \le \E_n(e^{t(Z_n - n\epsilon)}) = e^{-t n \epsilon} \E_n(e^{tZ_n}).
%\]
%It follows easily that 
%\[
%\limsup_{n\to\infty} \frac{1}{n} \log \mu_n(Z_n > n \epsilon) \le - \sup_{0<t<\eta}\{ t \epsilon - c(t)\}
%\]
%and the right hand side is strictly negative by the strict convexity of $c$ and the fact that $c'(0)= 0$.
%\end{proof}

We also need the following asymptotic result. This can be seen as a weak version of  the Hwang Quasi-power Theorem mentioned in the introduction.

\begin{proposition} Let $j \in J$.
There exists an  open real neighbourhood $U$ with $0 \in U$, a real analytic function $\sigma: U \to \R_{> 0}$ such that the following holds. For any $t_0 \in U$, there is a positive constant $C_{t_0}$  such that
\[
\sum_{w(x) < Q} e^{t_0 \varphi_j (x)}  \sim C_{t_0} e^{\sigma(t_0) Q}
\]
as $Q\to\infty$. Furthermore $\sigma$ is strictly convex on $U$.
\end{proposition}

\begin{proof}
We begin by applying the Implicit Function Theorem to find a real, open neighbourhood $U$ of $t=0$ and a real analytic function $\sigma : U \to \R$ such that $\lambda(\sigma(t),t) =1$ for all $t \in U$.

The Implicit Function Theorem also implies that
\[
\sigma'(0) =  - \frac{\lambda_t(1,0)}{\lambda_s(1,0)} = 0
\]
and 
\[
\sigma''(0)= - \frac{\lambda_{t,t}(1,0)}{\lambda_s(1,0)} > 0
\]
by Lemma \ref{lem.deriv}.
In particular,  $\sigma$ is strictly convex.
%Choose $\epsilon, \epsilon' >0$ sufficiently small so that for all $|t_0| < \epsilon$ there is $|s_0-1|  < \epsilon'$ such that $\mathcal{L}_{s_0,t_0}$ has a real simple maximal eigenvalue $\lambda(s_0,t_0) = 1$ (note that $s_0$ is real). 
Then for any fixed $t_0\in U$ we can find $s_0 = \sigma(t_0)$ and a neighbourhood of $s_0$ such that
\begin{equation} 
\zeta(s,t_0) = \frac{ \lambda(s,t_0)Q(s,t_0)\textbf{1}(x_0)}{1-\lambda(s,t_0)} + R(s,t_0)
\end{equation}
in this neighbourhood (here $R(s,t_0)$ is analytic in $s$). In particular, $\zeta(s,t_0)$ has a simple pole $s = s_0$. Furthermore, the expression
\[
\zeta(s,t_0) = \mathcal{L}_{s,t}^\sharp \circ \sum_{n\ge 0} \mathcal{L}_{s,t_0}^n \1(x_0) = \sum_{x \in X} e^{-sw(x) + t_0 \varphi_j(x)}
\]
shows us that $\zeta(s,t_0)$ does not have any other poles in the line $\Re(s) = s_0$. To see this, note that the operator $\mathcal{L}_{s,t_0}$ has simple maximal real eigenvalue $1$ when $s = s_0$. Then as in the proof of Lemma \ref{lem.spectrum} we see that $\mathcal{L}_{s_0 + i\xi, t_0}$ can not have $1$ as an eigenvalue as this would imply that $w$ takes values in a lattice which contradicts Assumption \ref{assump:nonarith}.

We also see that
\[
\zeta(s,t_0) =  \frac{\lambda_s(\sigma(t_0),t_0)^{-1} \lambda(s,t_0)Q(s,t_0)\textbf{1}(x_0)}{\sigma(t) - s} + \widetilde{R}(s)
\]
 where $\widetilde{R}$ is analytic in a neighbourhood of $s_0$. Note that $\lambda_s(\sigma(t_0),t_0) > 0$ by the same argument as in previous lemma.
 To summarise we have shown that $\zeta(s,t_0) $ is analytic in a neighbourhood of $\mathfrak{Re}(s) \ge s_0$ apart from a simple pole with positive residue at $s = s_0$.
 
Then we see by Proposition \ref{prop.tau} that there exists a positive constant $C_{t_0}$ depending on $t_0$ such that
\[
\sum_{w(x) < Q} e^{t_0 \varphi_j(x)}  \sim C_{t_0} e^{\sigma(t_0) Q}
\]
as $Q\to\infty$. 
\end{proof}

We are now ready to prove our large deviation theorem.

\begin{proof}[Proof of Theorem \ref{thm.largedev}]
Combining the above results show that there is a neighbourhood  $U$ of $0$ and an analytic function $\sigma : U \to\R$  such that if $t \in U$ then 
\[
\lim_{n \to\infty} \frac{1}{n} \log \left( \frac{1}{\#\{x \in X : w(x) < n \} } \sum_{w(x) < n } e^{t \varphi_j(x)} \right) = \sigma(t) - \sigma(0).
\]
Here we are taking a discrete limit as $n$ runs through the integers (despite the fact the limit exists as a continuous limit in $Q$). We are doing this so that we can apply Lemma \ref{lem.localldp} (note that $\sigma''(0) > 0$) to deduce that
\[
\limsup_{n \to\infty} \frac{1}{n} \log \left( \frac{1}{\#\{x \in X : w(x) < n \}} \#\left\{ x\in X: w(x) < n, \, \left| \frac{N_j(x)}{n} - \Lambda_j \right| > \epsilon \right\}\right) < 0.
\]
Here the random variables are $Z_n(x) = \varphi(x)$ if $w(x) < n$ and $Z_n(x) = 0$ otherwise and the measures $\mu_n$ are the uniform counting measures on $\{x \in X: w(x) < n \}$.
To conclude the proof of the theorem we just need to note that 
\[
\#\left\{ x \in X : w(x) < Q \, \left| \frac{N(x)}{Q} - \Lambda_j \right| > \epsilon  \right\} \le \#\left\{ x \in X : w(x) < \lfloor Q \rfloor + 1 \, \left| \frac{N(x)}{\lfloor Q \rfloor + 1} - \Lambda_j \right| > \frac{\epsilon}{2}  \right\} 
\]
for all $Q$ sufficiently large. It is then easy to deduce the required exponential decay in the continuous limit $Q \to \infty$.
\end{proof}

\subsection{Central limit theorems}\label{subsec:CLT}

To prove our central limit theorem we follow method presented in Section 3.7 of \cite{CanPol2}. We include all details for the convenience of the reader.

We begin by defining the following function.
\begin{definition}
For each $s \in \C, t \in \C^d$ we formally define
\[
\eta(s,t) = \sum_{x \in X}  e^{ -s w(x) + \langle t, \overline{\varphi}(x) \rangle}.
\]
Furthermore, for each $\widehat{p} = (p_1, \ldots, p_d)  \in \N^d$, we can formally define
\[
\eta_{\widehat{p}}(s) = \sum_{x \in X} \varphi_{\widehat{p}}(x) e^{-sw(x)}
\]
where $\varphi_{\widehat{p}}(x) = \varphi_1^{p_1}(x) \cdots \varphi_d^{p_d}(x)$.
\end{definition}

Both of these series converge to analytic functions providing the real part of $s$ is sufficiently large and $t$ is bounded.

\medskip
\noindent
{\bf Notation}.  Given $\widehat{p} =(p_1,\ldots,p_d) \in \mathbb N^d$, we write $|\widehat{p}| = p_1 +\cdots+p_d$.
\medskip

%Given $\widehat{p} =(p_1,\ldots,p_d) \in \N^d$ we will write $|\widehat{p}| = p_1 +\cdots+p_d$.

The first pole of $\eta(s,0)$ will appear on the real axis at $\lambda > 0$. The following gives the required properties of the complex function $\eta_{\w{p}}(s)$.

\begin{proposition}\label{prop.analyticity}
Given $\w{p} \in \mathbb{N}^d$ the function $\eta_{\widehat{p}}(s)$ is analytic in the plane $\mathfrak{R}(s) > 1$. Furthermore, $\eta_{\w{p}}(s)$ is analytic on $\mathfrak{R}(s) \ge 1$ apart from at $s = 1$. The nature   of the singularity at  $1$ depends on the  parity of $p = |\w{p}|$ as follows:\\
\noindent {\rm Case 1: $p$ is odd.} Then $\eta_{\w{p}}(s)$ has possibly finitely many  integer order poles at $s =1$ and is analytic otherwise. These poles have order at most $(p+1)/2$.\\
\noindent {\rm Case 2: $p$ is even}. In this case, there exists a positive definite, symmetric matrix $ \Sigma = (\sigma_{i,j})_{i,j=1}^n$
such that the following holds. For $s$ in a neighbourhood of $1$, we can write
\[
\eta_{\w{p}}(s) = \frac{R_{\w{p}}(s)}{(s-1)^{1+ \frac{p}{2}}}
\]
where each $R_{\w{p}}(s)$ is analytic and
\[
R_{\w{p}}(1) = C\left(\frac{p}{2}\right)! \sum_{i_1, \ldots, i_{p} } \sigma_{l(i_1),l(i_2)} \sigma_{l(i_3),l(i_4)} \cdots \, \sigma_{l(i_{p-1}), l(i_{p})}
\]
where:
\begin{enumerate}
\item
the sum over $i_1, \ldots, i_{p}$ is over the partition the numbers $1,\ldots, p$ into disjoint pairs  $(i_1,i_2), \ldots, (i_{p-1},i_{p})$; and 
\item $l : \{1, \ldots, p \} \to \{1,2,\ldots,p\}$ sends the set $\{1,\ldots,p_1\}$ to $1$,  the set $\{p_1+1, \ldots, p_1+p_2\}$ to $2$ and continues in this way until finally sending $\{p_1+\cdots + p_{d-1} + 1, \ldots, p\}$ to $d$.
\end{enumerate}
\end{proposition}

\begin{proof}[Proof of Proposition \ref{prop.analyticity}]
We first show that $\eta_{\w{p}}$ is analytic at $\mathfrak{Re}(s) \ge 1$ other than $s=1$. To do so we notice that
\[
\eta(s,t) = \mathcal{L}_{s,t}^\sharp \circ \sum_{n\ge0} \mathcal{L}_{s,t}^n \textbf{1}(x_0).
\]

Using Lemma \ref{lem.spectrum}, we see that $\eta(s,t)$ has the right domain of analyticity and we can differentiate to get what we need. Studying the pole is the harder bit. For $(s,t)$ close to $(1,0)$ we can write 
\[
\mathcal{L}_{s,t}^n\textbf{1}(x_0) = \lambda(s,t)^n Q(s,t)\textbf{1}(x_0) + O(\theta^n)
\]
where $0 < \theta < 1$ is independent of $s,t$. It follows that
\[
\eta(s,t) = \frac{ \lambda(s,t)Q(s,t)\textbf{1}(x_0)}{1-\lambda(s,t)} + R(s,t)
\]
where $R(s,t)$ is multianalytic in a neighbourhood of $(1,0)$. Set $F(s,t) = \lambda(s,t)Q(s,t)\textbf{1}(x_0)$ and note that $F(1,0) > 0$.   We now  want to study the  partial derivatives
\[
\left. \frac{\partial^n}{\partial t_1^{n_1} \partial t_2^{n_2} \cdots \, \partial t_d^{n_d}} \right|_{(s,0)}  \frac{F(s,t)}{1-\lambda(s,t)}
\]
for $\w{n} = (n_1, n_2,\ldots,n_d) \in \N$ with $|\w{n}| = n_1 + n_2 + \cdots + n_d =n$.
Since we are only interested in the largest order poles, it suffices to study
\[
\left. \frac{\partial^n}{\partial t_1^{n_1} \partial t_2^{n_2} \cdots \, \partial t_d^{n_d}} \right|_{(s,0)}  \frac{1}{1-\lambda(s,t)}  \text{ where $n_1+n_2 + \cdots + n_d= n$.}
\]
Using Fa\`a di Bruno's formula we have that these derivatives are given by
\[
\sum_{\pi \in \Pi_n} \frac{|\pi|!}{(1-\lambda(s,0))^{|\pi|+1}} \prod_{B \in \pi} \lambda_B(s,0)
\]
where
\begin{enumerate}
\item $\Pi_n$ is the set of partitions of $\{1,\ldots,n\}$;
\item $|\pi|$ is the number of blocks  (i.e., cyles) in $\pi$;
\item in the product $B$ runs over the blocks in $\pi$; and,
\item $\lambda_B(s,0)$ is the partial derivative of $\lambda$ over the block $B$: if $B = \{ b_1,\ldots, b_k\}\subset \{ 1, \ldots, n\} $ then 
\[
\lambda_B(s,0,0) = \frac{\partial^n \lambda}{\partial t_1^{l_1} \partial t_2^{l_2} \cdots \, \partial t_d^{l_d}} (s,0)
\]
where $l_1 = \# (B \cap \{ 1,\ldots, n_1\}), l_2= \# (B \cap \{ n_1 + 1, \ldots, n_1+n_2\}), \ldots, l_d = \#(B \cap \{n_1 + \cdots + n_{d-1} +1, \ldots, n \}$.
\end{enumerate}
Now notice that by Lemma \ref{lem.deriv} for any block $B$ of length $1$, $\lambda_B(1,0) = 0$. Hence, the Fa\`a di Bruno formula shows that when we are searching for the poles of highest order we can ignore terms coming from partitions that contain blocks of a single number.

 It therefore follows that when $n$ is odd, the highest order pole coming from the derivative above has order at most $|\pi|+1= (n - 3)/2 + 1+1  = (n+1)/2$. 
 Indeed, one considers  a cycle of  order $3$ and $(n - 3)/2$ transpositions.

When $n$ is even the highest order poles come from partitions of $\{ 1,\ldots, n\}$ into pairs. We deduce that, in this case, the highest order poles are of the form
\[
\frac{(n/2)!}{(1-\lambda(s,0))^{n/2 + 1}} \sum_{\pi \in \Pi_n(2)} \prod_{B \in \pi} \lambda_B(s,0)
\]
where $\Pi_n(2)$ represents all partitions of $\{1,\ldots, n\}$ into blocks of size $2$. Using  Lemma \ref{lem.deriv} we can then write
\[
\frac{(n/2)!}{(1-\lambda(s,0))^{n/2}} \sum_{\pi \in \Pi_n(2)} \prod_{B \in \pi} \lambda_B(s,0) = \frac{(-\lambda_s(1,0))^{n/2}(n/2)!}{(1-\lambda(s,0))^{n/2 +1}} \sum_{i_1,\ldots, i_n} \sigma_{l(i_1),l(i_2)}\cdots \sigma_{l(i_{n-1}),l(t_{n})} + g(s)
\]
where $g(s)$ has poles of integer orders strictly less that $n/2 + 1$ and $l : \{1, \ldots, p \} \to \{1,2,\ldots,n\}$ sends the set $\{1,\ldots,n_1\}$ to $1$,  the set $\{n_1+1, \ldots, n_1+n_2\}$ to $2$ and continues in this way until finally sending $\{n_1+\cdots + n_{d-1} + 1, \ldots, n\}$ to $d$ (as in the statement of the proposition).

To conclude the proof we sum over $0 \le n_1 \le p_1, 0 \le n_2 \le p_2$. To see that when $|\w{p}|$ is odd, $\eta_{\w{p}}$ has the required pole structure and when $|\w{p}| = p$ is even
\[
\eta_{\w{p}}(s) =  \frac{F(s,0)(-\lambda_s(1,0))^{p/2}(p/2)!}{(1-\lambda(s,0))^{p/2 +1}} \sum_{i_1,\ldots, i_p} \sigma_{l(i_1),l(i_2)}\cdots \sigma_{l(i_{p-1}),l(t_{p})} + G(s)
\]
where $G$ is analytic other than integer order poles of order at most $p/2$. To conclude we note that
\[
\frac{F(1,0)(-\lambda_s(1,0))^{p/2}(p/2)!}{(1-\lambda(s,0))^{p/2 +1}} = \frac{F(1,0)(p/2)!}{(-\lambda_s(1,0)) (s-1)^{p/2 + 1}} + H(s)
\]
where $H(s)$ is analytic on $\mathfrak{R}(s) \ge 1$ except for finite integer order poles at $s=1$ of order at most $p/2$.
This concludes the proof with
\[
C = - \frac{F(1,0)}{\lambda_s(1,0)} > 0 \ \text{ and the } \ \sigma_{i,k}\ \text{ as defined above.}
\]
\end{proof}

\subsection{Deducing the Central Limit Theorem} \label{subsec:CLT2}
In this section we will employ the estimates on the complex function described in the previous section 
to deduce the central limit theorem.  The approach is to use the method of moments, following an approach inspired by Morris \cite{morris}.  

\begin{definition}
For each $d$-tuple $\w{p} =(p_1,\ldots,p_d) \in \N^d$, we define
\[
\pi_{\w{p}}(Q) = \sum_{w(x) < Q} \varphi_{\w{p}}(x) \ \text{ where } \varphi_{\w{p}}(x) = \varphi_1^{p_1}(x)\cdots\varphi_d^{p_d}(x).
\]
\end{definition}

We can now use Proposition \ref{prop.tau} in the proof of the following moment estimate.

\begin{proposition}\label{prop.even}
When $|\w{p}| = p$ is even, we have that
\[
\lim_{Q\to\infty} \frac{1}{\#\{x \in X: w(x) <Q\}} \sum_{w(x) <Q} \left(\frac{\varphi_{\w{p}}(x)}{\sqrt{Q}}\right)^p = \sum_{i_1, \ldots, i_{p} } \sigma_{\pi(i_1),\pi(i_2)} \sigma_{\pi(i_3),\pi(i_4)} \cdots \sigma_{\pi(i_{p-1}), \pi(i_{p})}.
\]
\end{proposition}

\begin{proof}
When all of the $p_i$ are even we can apply Proposition \ref{prop.tau} to immediately deduce the result.
When some of the $p_i$ are odd we have to work harder. There are a further $2$ sub-cases. \\

\noindent 
\textit{Sub-case $1$: $p/2$ is even.} When this is the case we define $G_1(s), G_2(s)$ and $G_3(s)$ by 
\begin{align*}
 \sum_{x \in X} \left(\varphi_{2\w{p}}(x) + w(x)^p\right) e^{-sw(x)},  \  \sum_{x \in X} \left(\varphi_{\w{p}}(x)+ w(x)^{p/2}\right)^2 e^{-sw(x)} \ \text{ and } \ \sum_{x \in X} w(x)^{p/2} \varphi_{\w{p}}(x) e^{-sw(x)}
\end{align*}
respectively.
We note that $G_2 = G_1 + 2G_3$.
Using Proposition \ref{prop.tau} in combination with Proposition \ref{prop.analyticity} we see that
\[
\frac{1}{\#\{x \in X: w(x) < Q\}}\sum_{w(x) < Q} \varphi_{2\w{p}}(x) + w(x)^p \sim \frac{Q^p\left( R_{2\w{p}}(1) + R_{0,0}(1)(p!/(p/2)!)\right)}{Cp!}
\]
as $Q\to\infty$. We also have that (since $p/2$ is even)
\[
\eta_{\w{p}}^{(p/2)}(s) = \sum_{x \in X} w(x)^{p/2} \varphi_{\w{p}}(x)e^{-sw(x)} = G_3(x).
\]
Then using that $G_2 = G_1 + 2G_3$ we see that
\[
G_2(s) = \frac{R_{2\w{p}}(1) + R_{0,0}(1)(p!/(p/2)!) + 2 R_{\w{p}}(1)(p!/(p/2)!)}{(s-\lambda)^{1+p}} + f(s)
\]
where $f(s)$ is analytic other than integer poles of order at most $p$. Therefore
\[
\frac{1}{\#\{x \in X: w(x) < Q\}} \sum_{w(x) <Q} \left(\varphi_{\w{p}}(x)+ w(x)^{p/2}\right)^2 
\]
grows asymptotically like
\[
\frac{R_{2\w{p}}(1) + R_{0,0}(1)(p!/(p/2)!) + 2 R_{p_1,p_2}(1)(p!/(p/2)!)}{C p!} \, Q^p.
\]
This implies that
\[
\frac{1}{\#\{x \in X: w(x) < Q\}}\sum_{w(x) < Q} w(x)^{p/2} \varphi_{w{p}}(x)\sim \frac{R_{\w{p}}(1) Q^{p}}{C(p/2)!} .
\]
We want the same expression but with $w(x)^{p/2}$ replaced with $Q^{p/2}$. We now remove this weighting term. Note that
\[
\sum_{w(x) < Q} w(x)^{p/2} \varphi_{\w{p}}(x) = \int_0^Q t^{p/2} \ d\pi_{\w{p}}(t) = Q^{p/2}\pi_{\w{p}}(Q) - \frac{p}{2} \int_0^Q t^{p/2 -1} \pi_{\w{p}}(t) \ dt.	
\]
Without loss of generality we can assume that $p_1,\ldots,p_{2k}$ are odd and $p_{2k}+1, \ldots, p_d$ are even for some $k \le \lfloor p/2 \rfloor$.

We now would like to use the elementary inequality for real numbers $x_1,\ldots,x_{2k}$,
\[
\sum_{B \subset \{1,\ldots,2k\}, |B| = k} x_B^2 \ge |x_1 \cdots x_{2k}|
\]
where for $ B =\{b_1,\ldots, b_k\}  \subset \{1,\ldots,2k\}$ we set $x_B = x_{b_1}\cdots x_{b_k}$.

Using this we see that
\[
|\pi_{\w{p}}(t)| \le \sum_{B \subset \{1,\ldots,2k\}, |B| = k} \pi_{B(\w{p})}(t) 
\]
where $B(\w{p}) \in \N^d$ is the vector with entries $p_i$ for $i \nin \{1,\ldots,2k\}$, $p_i + 1$ if $i \in B \cap \{1,\ldots,2k\}$ and $p_i-1$ if $i \in  \{1,\ldots,2k\} \backslash B$.
Since all of the entries in $B(\w{p})$ are even and $|B(\w{p})| = |\w{p}| = p$ we can apply Proposition \ref{prop.tau} to deduce that
\[
\sum_{B \subset \{1,\ldots,2k\}, |B| = k} \pi_{B(\w{p})}(t)  = O(t^{p/2}e^t)
\]
as $t\to\infty$. It therefore follows that
\[
\int_0^Q t^{p/2 -1} \pi_{p_1,p_2}(t) \ dt= o\left( Q^{p/2} e^Q\right)
\]
as $Q\to\infty$.
This implies the required asymptotic in this sub-case.\\

\noindent
\textit{Sub-case $2$: $p/2$ is odd.} In this case we define $H_1(s), H_2(s)$ and $H_3(s)$ by
\begin{align*}
\sum_{x \in X} \left(\varphi_{2\w{p}}(x)+ w(x)^p\right) e^{-sw(x)}, \ \sum_{x \in X} \left(\varphi_{\w{p}}(x) + w(x)^{p/2}\right)^2 e^{-sw(x)} \ \text{ and } \ \sum_{x \in X} -w(x)^{p/2} \varphi_{\w{p}}(x) e^{-sw(x)}
\end{align*}
respectively. Similarly to before $H_2 = H_1 - 2 H_3$. Following the same argument as before, we deduce that
\[
\frac{1}{\#\{x \in X: w(x) < Q\}}\sum_{w(x) <Q} w(x)^{p/2} \varphi_{\w{p}}(x)\sim \frac{R_{p_1,p_2}(1) Q^{p}}{C(p/2)!} .
\]
Again, we can use the Stiltjes integral as before to change this expression into the required asymptotic.
\end{proof}

We now handle the odd sum case.

\begin{proposition}\label{prop.odd}
When $p = |\w{p}|$ is odd,  we have that
\[
\lim_{Q\to\infty} \frac{1}{\#\{x \in X: w(x) <Q\}} \sum_{w(x) <Q} \left(\frac{\varphi_{\w{p}}(x)}{\sqrt{Q}}\right)^p = 0.
\]
\end{proposition}

\begin{proof}
As before we define $K_1(s), K_2(s)$ and $K_3(s)$ by 
\begin{align*}
 \sum_{x \in X} \left(\varphi_{2\w{p}}(x) + w(x)^p\right) e^{-sw(x)}, \
 \sum_{x \in X} \left(\varphi_{\w{p}}(x)+ w(x)^{p/2}\right)^2 e^{-sw(x)}, \text{ and } \
 \sum_{x \in X} w(x)^{p/2} \varphi_{\w{p}}(x)(x) e^{-sw(x)}
\end{align*}
respectively
and note that $K_2 = K_1 + 2K_3$ and 
\[
K_3(s) = \eta_{2\w{p}}(s) - \eta_0^{(p)}(s) = \frac{g(s)}{(s-1)^{1+p}} + f(s)
\]
where $g(1) > 0$ and $f,g$ are analytic.
By the Tauberian theorem we deduce that
\[
\frac{1}{\#\{x\in X: w(x) < Q\}} \sum_{w(x)<Q} \varphi_{2\w{p}}(x)+ w(x)^p \sim \frac{g(1) Q^p}{p!}.
\]

We now calculate
\[
\eta_{\w{p}}^{\left( \frac{p+1}{2}\right)}(s) = (-1)^{\left(\frac{p+1}{2}\right)} \sum_{w(x) < Q} w(x)^{\left( \frac{p+1}{2}\right)} \varphi_{\w{p}}(x) e^{-sw(x)}.
\]
Now using the identity
\[
\int_0^\infty t^{-1/2} e^{-tx} \ dt = \sqrt{\pi} x^{-1/2}
\]
for $x > 0$ we see that
\[
K_3(s) = \frac{(-1)^{\frac{p+1}{2}}}{\sqrt{\pi}} \int_0^\infty \frac{\eta_{\w{p}}^{\left( \frac{p+1}{2}\right)}(s+t)}{\sqrt{t}} \ dt
\]
and hence $K_3$ is analytic except for a pole of order at most $p+1$ at $s=1$. Now, we can write
\[
\eta_{\w{p}}^{\left( \frac{p+1}{2}\right)}(s) = \sum_{i=1}^{p+1} \frac{a_i}{(s-1)^i} + h(s)
\]
where $a_i \in \C$, $h$ is analytic in $\mathfrak{R}(s) \ge 1$. Then using the identity
\[
\int_0^\infty \frac{1}{(s+t - 1)^i \sqrt{t}} \ dt = \frac{\pi (2i -2)!}{2^{2i-1}(i-1)!^2} \frac{1}{(s-1)^{k - \frac{1}{2}}}
\]
we deduce that
\[
K_3(x) = \sum_{i=0}^{p+1} \frac{c_i}{(s-1)^{i-\frac{1}{2}}} + l(s)
\]
where $c_i \in \C$ and $l$ is analytic in the half plane. Then using that $k_2 = K_1 + 2K_2$ we deduce that
\[
K_2(x) = \sum_{i=0}^{p+1} \frac{a_i}{(s-1)^i} + \frac{b_i}{(s-1)^{i-\frac{1}{2}}} + r(s)
\]
and $a_{p+1}(1) = g(1)$. The Tauberian theorem now implies that
\[
\frac{1}{\#\{x \in X: w(x) < Q\}} \sum_{w(x) < Q} \left( \varphi_{\w{p}}(x)+ w(x)^{p/2} \right)^2  \sim \frac{g(1) Q^p}{p}
\]
and so
\[
\frac{1}{\#\{x \in X: w(x) <Q\}} \sum_{w(x) <Q} w(x)^{p/2} \varphi_{\w{p}}(x)= o(Q^p).
\]

To conclude the proof we need to remove the weighting term $w(x)^{p/2}$. To do so we set
\[
\phi(t) = t^{-p/2} \ \text{ and } \ \wt{\pi}(t) = \sum_{w(x) < t} w(x)^{p/2} w(x)^{p/2} \varphi_{\w{p}}(x)
\]
for $t>0$ and note that by the above $\wt{\pi}(t)$ is $O(t^pe^t)$. It follows that
\[
\int_0^Q \wt{\pi}(t) \phi'(t) \ dt = O( Q \cdot T^p e^t Q^{-1-p/2}) = O(Q^{p/2} e^Q)
\]
as $Q\to\infty$.
However we also have that
\begin{align*}
\int_1^Q \wt{\pi}(t) \phi'(t) \ dt &= \int_1^Q\sum_{w(x) < Q} \varphi_{\w{p}}(x) w(x)^{p/2} \phi'(t) \ dt + O(1) \\
&= \sum_{w(x) < Q} \varphi_{\w{p}}(x) w(x)^{p/2} \int_{w(x)}^Q \phi'(t)  dt\\
 &= Q^{-p/2} \wt{\pi}(Q) - \pi_{\w{p}}(Q) + O(1)
\end{align*}
as $Q\to\infty$. Rearranging this and using our estimates above gives the required result.
\end{proof}

We are now ready to conclude the proofs of the results in the introduction.
\begin{proof}[proof of Theorem \ref{thm.mdclt:intro} and \ref{thm.largedev}]
We conclude using the method of moments. Indeed, the limits obtained in Proposition \ref{prop.even} and Proposition \ref{prop.odd} are the limits of the sequences of moments for the distributions considered in our main theorems. We have shown that these sequences converge to the moments of the Gaussian limit law with mean $0$ and co-variances given by the $\sigma_{i,k}$. Along with  Proposition \ref{prop.var} and Proposition \ref{prop.multivar} which characterise non-degeneracy, this concludes the proof using the method of moments.
\end{proof}

%%%%%%%%%%%%%%%%%%%%%%%%%%%%%%%%%%%%%%%%%%%%%%%%%%%%%%%%%%%%%%%%%%%%%%%%%%%%%%%

\subsection*{Open access statement}
For the purpose of open access, the authors have applied a Creative Commons Attribution (CC BY) licence to any Author Accepted Manuscript version arising from this submission.

\bibliographystyle{alpha}
\bibliography{CLT}

\begin{thebibliography}{PLRD84}

\bibitem[AN93]{Arnoux-Nogueira}
P.~Arnoux and A.~Nogueira.
\newblock Mesures de {G}auss pour des algorithmes de fractions continues
  multidimensionnelles.
\newblock {\em Ann. Sci. \'Ecole Norm. Sup. (4)}, 26(6):645--664, 1993.

\bibitem[Bal00]{Baladi}
V.~Baladi.
\newblock {\em Positive transfer operators and decay of correlations},
  volume~16 of {\em Advanced Series in Nonlinear Dynamics}.
\newblock World Scientific Publishing Co., Inc., River Edge, NJ, 2000.

\bibitem[BD22]{BetDra}
S.~Bettin and S.~Drappeau.
\newblock Limit laws for rational continued fractions and value distribution of
  quantum modular forms.
\newblock {\em Proc. Lond. Math. Soc. (3)}, 125(6):1377--1425, 2022.

\bibitem[Ber71]{Bernstein:71}
L.~Bernstein.
\newblock {\em The {J}acobi-{P}erron algorithm---{I}ts theory and application}.
\newblock Lecture Notes in Mathematics, Vol. 207. Springer-Verlag, Berlin-New
  York, 1971.

\bibitem[BG01]{BAG:01}
A.~Broise and Y.~Guivarch.
\newblock Exposants caract\'{e}ristiques de l'algorithme de {J}acobi-{P}erron
  et de la transformation associ\'{e}e.
\newblock {\em Ann. Inst. Fourier (Grenoble)}, 51(3):565--686, 2001.

\bibitem[BLV18]{BLV:18}
V.~Berth\'{e}, L.~Lhote, and B.~Vall\'{e}e.
\newblock The {B}run gcd algorithm in high dimensions is almost always
  subtractive.
\newblock {\em J. Symbolic Comput.}, 85:72--107, 2018.

\bibitem[Bro96]{broise}
A.~Broise.
\newblock Fractions continues multidimensionnelles et lois stables.
\newblock {\em Bull. Soc. Math. France}, 124(1):97--139, 1996.

\bibitem[Bru19]{Brun19}
V.~Brun.
\newblock En generalisation av kjedebr{\o}ken {I}.
\newblock {\em Skr. Vidensk.-Selsk. Christiana Math.-Nat. Kl.}, (6):1--29,
  1919.

\bibitem[Bru20]{Brun20}
V.~Brun.
\newblock En generalisation av kjedebr{\o}ken {II}.
\newblock {\em Skr. Vidensk.-Selsk. Christiana Math.-Nat. Kl.}, (6):1--24,
  1920.

\bibitem[Bru58]{BRUN}
V.~Brun.
\newblock Algorithmes euclidiens pour trois et quatre nombres.
\newblock In {\em Treizi\`eme congr\`es des math\`ematiciens scandinaves, tenu
  \`a {H}elsinki 18-23 ao\^ut 1957}, pages 45--64. Mercators Tryckeri,
  Helsinki, 1958.

\bibitem[BV05]{baladi-vallee}
V.~Baladi and B.~Vall\'{e}e.
\newblock Euclidean algorithms are {G}aussian.
\newblock {\em J. Number Theory}, 110(2):331--386, 2005.

\bibitem[CP24]{CanPol}
S.~Cantrell and M.~Pollicott.
\newblock Central limit theorems for green metrics on hyperbolic groups, 2024.

\bibitem[CP25]{CanPol2}
S.~Cantrell and M.~Pollicott.
\newblock Counting statistics for geodesics on flat surfaces, 2025.

\bibitem[Del54]{delange}
H.~Delange.
\newblock G\'en\'eralisation du th\'eor\`eme de {I}kehara.
\newblock {\em Ann. Sci. \'Ecole Norm. Sup. (3)}, 71:213--242, 1954.

\bibitem[dR95]{Ro}
P.~de~Rooij.
\newblock Efficient exponentiation using precomputation and vector addition
  chains.
\newblock In {\em Advances in Cryptology, EUROCRYPT '94 Proceedings}, volume
  950 of {\em Lecture Notes in Computer Science}, pages 389--399.
  Springer-Verlag, 1995.

\bibitem[Hen06]{HensleyBook}
D.~Hensley.
\newblock {\em Continued fractions}.
\newblock World Scientific Publishing Co. Pte. Ltd., Hackensack, NJ, 2006.

\bibitem[HH01]{hennion}
H.~Hennion and L.~Herv\'{e}.
\newblock {\em Limit theorems for {M}arkov chains and stochastic properties of
  dynamical systems by quasi-compactness}, volume 1766 of {\em Lecture Notes in
  Mathematics}.
\newblock Springer-Verlag, Berlin, 2001.

\bibitem[HJ68]{Heine1868}
E.~Heine and C.G.J. Jacobi.
\newblock {A}llgemeine {T}heorie der kettenbruch{\"{a}}hnlichen {A}lgorithmen,
  in welchen jede {Z}ahl aus drei vorhergehenden gebildet wird.
\newblock {\em J. Reine Angew. Math.}, 69:29--64, 1868.

\bibitem[Hwa98]{Hwang}
H.-K. Hwang.
\newblock On convergence rates in the central limit theorems for combinatorial
  structures.
\newblock {\em European J. Combin.}, 19(3):329--343, 1998.

\bibitem[Kat80]{kato}
T.~Kato.
\newblock {\em Perturbation Theory for Linear Operators}.
\newblock Springer-Verlag, Berlin, 1980.

\bibitem[KLL25]{kim2025}
D.~Kim, J.~Lee, and S.~Lim.
\newblock Euclidean algorithms are {G}aussian over imaginary quadratic fields,
  2025.

\bibitem[Lag93]{Lagarias:93}
J.~C. Lagarias.
\newblock The quality of the {D}iophantine approximations found by the
  {J}acobi-{P}erron algorithm and related algorithms.
\newblock {\em Monatsh. Math.}, 115(4):299--328, 1993.

\bibitem[LS25]{LeeSun}
J.~Lee and H.-S. Sun.
\newblock Dynamics of continued fractions and distribution of modular symbols.
\newblock {\em J. Eur. Math. Soc., published online first}, DOI
  10.4171/JEMS/1665, 2025.

\bibitem[May84]{mayer}
D.~Mayer.
\newblock Approach to equilibrium for locally expanding maps in {${\bf
  R}^{k}$}.
\newblock {\em Comm. Math. Phys.}, 95(1):1--15, 1984.

\bibitem[Mor15]{morris}
I.~D. Morris.
\newblock A short proof that the number of division steps in the euclidean
  algorithm is normally distributed, 2015.

\bibitem[Per07]{Perron:07}
O.~Perron.
\newblock Grundlagen f\"ur eine {T}heorie des {J}acobischen
  {K}ettenbruchalgorithmus.
\newblock {\em Math. Ann.}, 64(1):1--76, 1907.

\bibitem[PLRD84]{PLRD}
R.~Paysant-Le~Roux and E.~Dubois.
\newblock Une application des nombres de {P}isot \`a{} l'algorithme de
  {J}acobi-{P}erron.
\newblock {\em Monatsh. Math.}, 98(2):145--155, 1984.

\bibitem[PP90]{ParryPollicott}
W.~Parry and M.~Pollicott.
\newblock Zeta functions and the periodic orbit structure of hyperbolic
  dynamics.
\newblock {\em Ast\'{e}risque}, (187-188):268, 1990.

\bibitem[Rom08]{Romik}
D.~Romik.
\newblock The dynamics of {P}ythagorean triples.
\newblock {\em Trans. Amer. Math. Soc.}, 360(11):6045--6064, 2008.

\bibitem[Ros54]{Rosen}
D.~Rosen.
\newblock A class of continued fractions associated with certain properly
  discontinuous groups.
\newblock {\em Duke Math. J.}, 21:549--563, 1954.

\bibitem[Sch73]{Schweiger:73}
F.~Schweiger.
\newblock {\em The metrical theory of {J}acobi-{P}erron algorithm}.
\newblock Lecture Notes in Mathematics, Vol. 334. Springer-Verlag, Berlin-New
  York, 1973.

\bibitem[Sch79]{Schweiger79}
F.~Schweiger.
\newblock A modified {J}acobi-{P}erron algorithm with explicitly given
  invariant measure.
\newblock In {\em Ergodic theory ({P}roc. {C}onf., {M}ath. {F}orschungsinst.,
  {O}berwolfach, 1978)}, volume 729 of {\em Lecture Notes in Math.}, pages
  199--202. Springer, Berlin, 1979.

\bibitem[Sch00]{Schweiger:00}
F.~Schweiger.
\newblock {\em Multidimensional continued fractions}.
\newblock Oxford Science Publications. Oxford University Press, Oxford, 2000.

\bibitem[Sch09]{Sch}
T.~A. Schmidt.
\newblock Rosen fractions and veech groups, an overly brief introduction.
\newblock {\em Actes des rencontres du CIRM}, 1(1):61--67, 3 2009.

\bibitem[Vee78]{Veech:78}
W.~A. Veech.
\newblock Interval exchange transformations.
\newblock {\em J. Analyse Math.}, 33:222--272, 1978.

\bibitem[Wen18]{Wen2018}
J.~Wen.
\newblock {\em Lattice Reduction and Its Applications in Wireless Sensors
  Network}, pages 1--4.
\newblock Springer International Publishing, Cham, 2018.

\end{thebibliography}

\end{document}